%% file: main.tex
\documentclass{article}

    \usepackage[final]{neurips_2022}

\usepackage{amsmath,amsthm,amssymb}
\usepackage[utf8]{inputenc} 
\usepackage[T1]{fontenc}    
\usepackage{hyperref}       
\usepackage{url}            
\usepackage{booktabs}       
\usepackage{amsfonts}       
\usepackage{nicefrac}       
\usepackage{microtype}      
\usepackage{xcolor}         
\usepackage{mathtools}
\usepackage{caption}
\usepackage{subcaption}
\usepackage{mathabx}
\usepackage[algo2e,ruled,linesnumbered]{algorithm2e}
\SetKwInput{KwInput}{Input}                
\SetKwInput{KwOutput}{Output}              

\newcommand{\argmax}{\operatornamewithlimits{\arg\!\max}}

\newcommand{\B}{\mathcal{B}}  
\newcommand{\C}{\mathcal{C}}  
\newcommand{\E}{\operatornamewithlimits{\mathbb{E}}}  
\newcommand{\F}{\mathcal{F}}  
\DeclareMathOperator*{\essup}{\text{ess~sup}}
\DeclareMathOperator*{\esinf}{\text{ess~inf}}
\newcommand{\inv}{^{-1}}
\renewcommand{\L}{\mathcal{L}} 
\renewcommand{\P}{\mathcal{P}} 
\newcommand{\prob}{\mathbb{P}}  
\newcommand{\R}{\mathbb{R}} 
\newcommand{\SC}{\mathcal{SC}}  
\newcommand{\X}{\mathcal{X}} 
\newcommand{\Y}{\mathcal{Y}} 

\newcommand{\edit}[1]{#1}

\newcommand{\truepositive}{\text{TP}}
\newcommand{\truenegative}{\text{TN}}
\newcommand{\falsepositive}{\text{FP}}
\newcommand{\falsenegative}{\text{FN}}

\renewcommand{\hat}{\widehat} 
\renewcommand{\tilde}{\widetilde} 

\newcommand{\sigmafield}{\mathcal{F}}
\newcommand{\ind}{\mathbf{1}}

\newtheoremstyle{myThm}   
     {\topsep}                          
     {\topsep}                          
     {\itshape}                         
     {}                                      
     {\sffamily\bfseries}           
     {.}                                     
     {.5em}                              
     {}                                      

\newtheoremstyle{myRem}   
     {\topsep}                          
     {\topsep}                          
     {}                         
     {}                                      
     {\sffamily\bfseries}           
     {.}                               
     {.5em}                              
     {}                                      
     
\newtheoremstyle{myDef}   
     {\topsep}                          
     {\topsep}                          
     {\itshape}                         
     {}                                      
     {\sffamily\bfseries}           
     {.}                                 
     {.5em}                              
     {}                                      

\newcounter{thm}     
\theoremstyle{myThm}
\newtheorem{theorem}[thm]{Theorem}
\newtheorem{lemma}[thm]{Lemma}

\newtheorem{corollary}[thm]{Corollary}

\newtheorem{assumption}[thm]{Assumption}

\newtheorem{innercustomthm}{Theorem}
\newenvironment{customthm}[1]
  {\renewcommand\theinnercustomthm{#1}\innercustomthm}
  {\endinnercustomthm}
\newtheorem{innercustomlemma}{Lemma}
\newenvironment{customlemma}[1]
  {\renewcommand\theinnercustomlemma{#1}\innercustomlemma}
  {\endinnercustomlemma}
\newtheorem{innercustomcorr}{Corollary}
\newenvironment{customcorollary}[1]
  {\renewcommand\theinnercustomcorr{#1}\innercustomcorr}
  {\endinnercustomcorr}

\theoremstyle{myRem}

\newtheorem{example}[thm]{Example}

\theoremstyle{myDef}
\newtheorem{definition}[thm]{Definition}

\title{Optimal Binary Classification Beyond Accuracy}

\author{%
  Shashank Singh\\
  Max Planck Institute for Intelligent Systems\\
  T\"ubingen, Germany \\
  \texttt{shashankssingh44@gmail.com}
  \And
  Justin Khim\thanks{The contributions in this paper were made prior to joining Amazon.} \\
  Amazon \\
  New York, NY \\
  \texttt{jkhim@amazon.com}
}

\begin{document}

\maketitle

\begin{abstract}
    \edit{The vast majority of statistical theory on binary classification characterizes performance in terms of accuracy. However, accuracy is known in many cases to poorly reflect the practical consequences of classification error, most famously in imbalanced binary classification, where data are dominated by samples from one of two classes. The first part of this paper derives a novel generalization of the Bayes-optimal classifier from accuracy to any performance metric computed from the confusion matrix. Specifically, this result (a) demonstrates that stochastic classifiers sometimes outperform the best possible deterministic classifier and (b) removes an empirically unverifiable absolute continuity assumption that is poorly understood but pervades existing results. We then demonstrate how to use this generalized Bayes classifier to obtain regret bounds in terms of the error of estimating regression functions under uniform loss. Finally, we use these results to develop some of the first finite-sample statistical guarantees specific to imbalanced binary classification. Specifically, we demonstrate that optimal classification performance depends on properties of class imbalance, such as a novel notion called Uniform Class Imbalance, that have not previously been formalized. We further illustrate these contributions numerically in the case of $k$-nearest neighbor classification.}
\end{abstract}

\section{Introduction}
Many binary classification problems exhibit class imbalance, in which one of the two classes vastly outnumbers the other. 
Classifiers that perform well with balanced classes routinely fail for imbalanced classes, and developing reliable techniques for classification in the presence of severe class imbalance remains a challenging area of research~\citep{he2013imbalanced,krawczyk2016learning,fernandez2018learning}. 
Many practical approaches have been proposed to improve performance under class imbalance, including reweighting plug-in estimates of class probabilities~\citep{lewis1995evaluating}, resampling data to improve class imbalance~\citep{chawla2002smote}, or reformulating classification algorithms to optimize different performance metrics~\citep{dembczynski2013optimizing, fathony2019genericMetrics, joachims2005support}. 
Extensive discussion of practical methods for handling class imbalance are surveyed in the books of \citet{he2013imbalanced} and \citet{fernandez2018learning}.

Despite the pervasive challenge of class imbalance, our theoretical understanding of class imbalance is limited. 
The vast majority of theoretical performance guarantees for classification characterize classification accuracy or, equivalently, misclassification risk~\citep{mohri2018foundations}, which is typically an uninformative measure of performance for imbalanced classes. 
Under measures that are used with imbalanced classes in practice, such as precision, recall, $F_\beta$ scores, and class-weighted scores~\citep{van1974foundation, van1979information}, existing theoretical guarantees are limited to statistical consistency, in that the algorithm under consideration asymptotically optimizes the metric of choice~\citep{koyejo2014consistent,menon2013statistical,narasimhan2014statistical}; specifically, there is no finite-sample theory that would allow comparison of an algorithm's performance to that of other algorithms or to theoretically optimal performance levels.
Additionally, existing theory for classification does not explicitly model the effects of class imbalance, especially severe imbalance (i.e., as the proportion of samples from the rare class vanishes), and hence sheds little light on how severe imbalance influences optimal classification.

This paper provides two main contributions. 
First, in Section~\ref{sec:generalized_bayes}, we provide a novel characterization of classifiers optimizing general performance metrics that are functions of a classifier's confusion matrix.
This characterization generalizes a classical result, that the Bayes classifier optimizes classification accuracy, to a much larger class of performance measures, including those commonly used in imbalanced classification\edit{, while relaxing certain empirically unverifiable distributional assumptions that pervade existing such results}.
Interestingly, we show that, in general, a Bayes classifier always exists if one considers stochastic classifiers, but not if one considers only deterministic classifiers.
We then use this result 
to provide relative performance guarantees under these more general performance measures,
in terms of the error of estimating the class probability (regression) function under uniform ($\L_\infty$) loss.

This motivates our second main contribution: an analysis of $k$-nearest neighbor ($k$NN) classification under uniform loss. 
In doing so, we also propose an explicit model of a sub-type of class imbalance, which we call Uniform Class Imbalance, and we show that the $k$NN classifier behaves quite differently under Uniform Class Imbalance than under other sub-types of class imbalance. 
To the best of our knowledge, such sub-types of class imbalance have not previously been distinguished in either the theoretical or practical literature, and we hope that identifying such relevant features of imbalanced datasets may facilitate development of classifiers that perform well on specific imbalance problems of practical importance.
Collectively, these contributions provide some of the first finite-sample performance guarantees for nonparametric binary classification under performance metrics that are appropriate for imbalanced data and show how optimal performance depends on the nature of imbalance in the data.

\section{Related Work}
\label{sec:relatedWork}
Here, we discuss how our results relate to existing theoretical guarantees for imbalanced binary classification and prior analyses of $k$NN methods.

\subsection{Theoretical Guarantees for Imbalanced Binary Classification}
Statistical learning theory has studied classification extensively in terms of accuracy~\citep{mohri2018foundations}.
However, when classes are severely imbalanced, accuracy ceases to be an informative measure of performance~\citep{cortes2004auc}, necessitating guarantees in terms of other performance metrics.
Several papers have sought to address this~\citep{narasimhan2014statistical,narasimhan2015consistent,koyejo2014consistent,yan2018binary,wang2019consistent} \edit{by generalizing the Bayes optimal classifier, a well-known classifier that provably optimizes accuracy, to more general performance measures better reflecting the desiderata of imbalanced classification. Relatedly, several works have investigated relationships between these different performance measures and demonstrated that they differ essentially in how they determine the optimal threshold between the two classes~\citep{flach2003geometry,hernandez2013roc,flach2016classifier}. However, existing results make empirically unverifiable assumptions about the distribution of the data, leaving questions about their relevance to real data. We discuss these assumptions in detail in Section~\ref{sec:generalized_bayes}, where our main result, Theorem~\ref{thm:generalized_bayes}, leverages the idea of stochastic thresholding to relax these assumptions.}

Another body of closely related theoretical work studies Neyman-Pearson classification, which attempts to minimize misclassification error on one class subject to constraints on misclassification error on other classes, analogous to the approach of statistical hypothesis testing. While substantial theoretical guarantees do exist for Neyman-Pearson classification~\citep{rigollet2011neyman, tong2013plug, tong2016survey}, these focus on performance within the Neyman-Pearson framework, rather than under general performance measures as in our work, and we know of no work considering stochastic classification under the Neyman-Pearson framework. Interestingly, our use of stochastic classifiers in Theorem~\ref{thm:generalized_bayes} parallels classical results in hypothesis testing~\citep{lehmann2006testing}, and our proof of Theorem~\ref{thm:generalized_bayes} involves a reduction (Lemma~\ref{lemma:opt_problem} in the Appendix) of optimization of general classification performance measures to Neyman-Pearson classification.

Meanwhile, many practical approaches to handling class imbalance, such as class-weighting and resampling have been proposed, but the theoretical understanding of these methods is limited.
Class-weighting is a natural choice in applications where costs, \edit{or cost ratios~\citep{flach2003geometry}}, can be explicitly assigned and, in the case of binary classification, is statistically equivalent to threshold selection, which we discuss later in this paper~\citep{scott2012calibrated}.
In practice, resampling appears to be the most popular approach to handling class imbalance~\citep{he2013imbalanced}.
Undersampling the dominant class is straightforward and can provide computational benefits with little loss in statistical performance~\citep{fithian2014local}, while interest in oversampling rare classes, sometimes referred to as data augmentation, has grown with the advent of sophisticated generative models to produce additional data \citep{mariani2018bagan}.
However, the theoretical ramifications of oversampling techniques used for imbalanced classification, most commonly variants of SMOTE~\citep{chawla2002smote}, are poorly understood.

\subsection{\texorpdfstring{$k$}{k}NN Classification and Regression}
The $k$NN classifier is one of the oldest and most well-studied nonparametric classifiers~\citet{fix1951discriminatory}. 
Early theoretical results include, \citet{cover1967nearest}, who showed that the misclassification risk of the $k$NN classifier with $k = 1$ is at most twice that of the Bayes-optimal classifier, and \citet{stone1977consistent}, who showed that the $k$NN classifier is Bayes-consistent if $k \to \infty$ and $k/n \to 0$.
Extensive literature on the accuracy of $k$NN classification has since developed~\citep{devroye1996probabilistic,gyorfi2002distribution,samworth2012optimal,chaudhuri2014rates,gottlieb2014efficient,biau2015lectures,gadat2016classification,doring2018rate,kontorovich2015bayes,gottlieb2018near,cannings2019local,hanneke2020universal}.

Rather than accuracy bounds for $k$NN classification, the bounds on uniform error we present in Section~\ref{sec:uniform_error} are most closely related to risk bounds for $k$NN regression, of which the results of \cite{biau2010rates} are representative.
\cite{biau2010rates} gives convergence rates for $k$NN regression in $\L_2$ risk, weighted by the covariate distribution, in terms of noise variance and covering numbers of the covariate space. While closely related to our bounds on uniform ($\L_\infty$) risk, their results differ in at least three main ways. 
First, minimax rates under $\L_\infty$ risk are necessarily worse than under $\L_2$ risk by a logarithmic factor, as implied by our lower bounds. 
Second, the fact that \citet{biau2010rates} use a risk that is weighted by the covariate distribution allows them to avoid our assumption that the covariate density is lower bounded away from \(0\), whereas, the lower boundedness assumption is unavoidable under $\L_\infty$ risk.
Finally, \citet{biau2010rates} assume additive noise with finite variance; Bernoulli noise is crucial for us to model severe class imbalance.

\edit{Extensive r}esearch on $k$NN for imbalanced classification has focused on algorithmic modifications, which are surveyed by \cite{fernandez2018learning}.
Examples include prototype selection \citep{liu2011class, lopez2014addressing, vluymans2016eprennid},
and gravitational methods~\citep{cano2013weighted, zhu2015gravitational}.
We are aware of no statistical guarantees exist for such methods.

\section{Setup and Notation}
\label{sec:Setup}
Let \((\X, \rho)\) be a separable metric space, and let \(\Y = \{0, 1\}\) denote the set of classes. For any $x \in \X$ and $\epsilon > 0$, $B(x, \epsilon) := \{z \in \X : \rho(x, z) < \epsilon\}$ denotes the open radius-$\epsilon$ ball around $x$.
Consider \(n\) independent samples \((X_{1}, Y_{1}), \ldots, (X_{n}, Y_{n})\) drawn from a distribution \(P_{X, Y}\) on \(\X \times \Y\) with marginals $P_X$ and $P_Y$.
For positive sequences $\{a_n\}_{i = 1}^\infty$ and $\{b_n\}_{i = 1}^\infty$, $a_n \asymp b_n$ means $\liminf_{n \to \infty} a_n/b_n > 0$ and $\limsup_{n \to \infty} a_n/b_n < \infty$.

To optimize general performance metrics, we must consider stochastic classifiers. 
Formally, letting $\B := \{Y \sim \text{Bernoulli}(p) : p \in [0, 1]\}$ denote the set of binary random variables, a stochastic classifier can be modeled as a mapping $\hat Y : \X \to \B$, where, for any $x \in \X$, $\E[\hat Y(x)]$ is the probability that the classifier assigns $x$ to class $1$. 
We use $\SC$ to denote the class of stochastic classifiers.

The true \emph{regression function} \(\eta^*: \X \to [0, 1]\) is defined as $\eta^*(x) := \prob \left[Y = 1| X = x \right] = \E \left[ Y | X = x \right]$;
that is, given an instance $X_i$, the label $Y_i$ is Bernoulli-distributed with mean $\eta(X_i)$.
As we show in the next section, an optimal classifier can always be written in terms of the true regression function $\eta$, motivating estimates $\hat \eta : \X \to [0, 1]$ of $\eta$. Such estimates $\hat \eta$ are referred to as ``regressors''.

\section{Optimal Classification Beyond Accuracy}
\label{sec:generalized_bayes}
A famous result states that classification accuracy is maximized by the  ``Bayes'' classifier
\begin{equation}
    \hat Y(x) \sim \text{Bernoulli}\left( 1\{\eta^*(x) > 0.5\} \right).
    \label{eq:bayes_classifier}
\end{equation}
\edit{Here, $\hat Y(x)$ is simply a constant (deterministic) random variable that takes either the value 0 or the value 1 with probability 1 (depending on $\eta^*(x)$). Our reason for writing Eq. (1) in this seemingly redundant way will become clear with Definition~\ref{def:regression_thresholding_classifier} below.}

Although $\eta^*$ is unknown in practice, this result is a cornerstone of the statistical theory of binary classification because it provides an optimal performance benchmark against which a classifier can be evaluated in terms of accuracy~\citep{devroye1996probabilistic,mitchell1997machine,james2013introduction}. 
As discussed previously, accuracy can be a poor measure of performance in the imbalanced case. 
Therefore, the main contribution of this section, provided in Theorem~\ref{thm:generalized_bayes} below, is to generalize this result to a broad class of classification performance measures, including those commonly used in imbalanced classification. 
First, we specify performance measures for which our results apply.

\subsection{Confusion Matrix Measures (CMMs)}
Nearly all measures of classification performance, including accuracy, precision, recall, $F_\beta$ scores, and others, can be computed from the confusion matrix, which counts the number of test samples in each $(\text{true class, estimated class})$ pair.
Formally, let $\C := \{C \in [0, 1]^{2 \times 2} : C_{1,1} + C_{1, 2} + C_{2, 1} + C_{2,2} = 1\}$ denote the set of all possible binary confusion matrices.
Given a classifier $\hat Y$, the \emph{confusion matrix} $C_{\hat Y} \in \C$ and \emph{empirical confusion matrix} $\hat C_{\hat Y} \in \C$ are given by
\begin{equation}
    C_{\hat Y} =
    \begin{bmatrix}
        \truenegative_{\hat Y} & \falsepositive_{\hat Y} \\
        \falsenegative_{\hat Y} & \truepositive_{\hat Y}
    \end{bmatrix},
    \hat C_{\hat Y} =
    \begin{bmatrix}
        \hat\truenegative_{\hat Y} & \hat\falsepositive_{\hat Y} \\
        \hat\falsenegative_{\hat Y} & \hat\truepositive_{\hat Y}
    \end{bmatrix},
    \label{eq:true_and_empirical_CFs}
\end{equation}
wherein the true positive \edit{probability} $\truepositive_{\hat Y}$ and empirical true positive \edit{probability} $\hat\truepositive_{\hat Y}$ are given by
\begin{equation}
    \truepositive_{\hat Y} = \E \left[ \eta^*(X) \hat Y(X) \right],
    \widehat{\truepositive}_{\hat Y} = \frac{1}{n} \sum_{i = 1}^{n} Y_i \hat Y(X_i),
    \label{eq:true_and_empirical_TPs}
\end{equation}
and the true and empirical false positive ($\falsepositive_{\hat Y}$ and $\hat\falsepositive_{\hat Y}$), false negative ($\falsenegative_{\hat Y}$ and $\hat\falsenegative_{\hat Y}$), and true positive ($\truepositive_{\hat Y}$ and $\hat\truepositive_{\hat Y}$) probabilities are defined similarly. Note that the expectation in Eq.~\eqref{eq:true_and_empirical_TPs} is over randomness both in the data and in the classifier.

Intuitively, measures of a classifier's performance should improve as $\truenegative$ and $\truepositive$ increase and $\falsenegative$ and $\falsepositive$ decrease. 
We therefore define the class of Confusion Matrix Measures (CMMs) as follows:
\begin{definition}[Confusion Matrix Measure (CMM)]
    A function $M : \C \to \mathbb{R}$ is called a \emph{confusion matrix  measure (CMM)} if, for any confusion matrix
    \[C =
        \begin{bmatrix}
            \truenegative & \falsepositive \\
            \falsenegative & \truepositive
        \end{bmatrix} \in \C,
        \epsilon_1 \in [0, \falsepositive], \epsilon_2 \in [0, \falsenegative],
    \; \text{ we have } \; 
    M(C) \leq M \left(
        \begin{bmatrix}
            \truenegative + \epsilon_1 & \falsepositive - \epsilon_1 \\
            \falsenegative - \epsilon_2 & \truepositive + \epsilon_2
        \end{bmatrix}
    \right).
    \]
    \label{def:CMM}
\end{definition}
Essentially, correcting an incorrect classification should not reduce a CMM. This is true of any reasonable measure of classification performance, and hence analyzing CMMs allows us to obtain theoretical guarantees for all performance measures used in practice. Specifically, by evaluating their gradients in the directions
$\begin{bsmallmatrix}1 & -1\\0 & 0\end{bsmallmatrix}$ and $\begin{bsmallmatrix}0 & 0\\-1 & 1\end{bsmallmatrix}$, one can verify that most performance measures, such as weighted accuracy, precision, recall, $F_\beta$ scores, and Matthew's Correlation Coefficient are CMMs. We note that the area under receiver operating characteristic (AUROC) and area under precision-recall curve (AUPRC) are not CMMs because they evaluate ($\R$-valued) scoring functions rather than ($\{0,1\}$-valued) classification functions. However, both AUROC and AUPR are averages of CMMs computed at various classification thresholds, and, as we discuss in Appendix~\ref{app:AUROC}, our results for CMMs thus imply similar results for these measures. We next present our main result of Section~\ref{sec:generalized_bayes}, which generalizes the Bayes classifier~\eqref{eq:bayes_classifier} to arbitrary CMMs.

\subsection{Generalizing the Bayes Classifier}
The Bayes classifier thresholds the regression function deterministically at the value $0.5$. 
The following generalizes this to a stochastic threshold:
\begin{definition}[Regression-Thresholding Classifier (RTC)]
    A classifier $\hat Y : \X \to \B$ is called a \emph{regression-thresholding classifier} (RTC) if, for some $p, t \in [0, 1]$ and $\eta : \X \to [0, 1]$,
    \[\hat Y(x) \sim \text{Bernoulli}(p \cdot 1\{\eta(x) = t\} + 1\{\eta(x) > t\}),
      \quad \text{ for all } x \in \X.\]
    In the sequel, we will denote such classifiers $\hat Y_{p,t,\eta}$, and refer to the pair $(t, p)$ as the \emph{threshold}.
    \label{def:regression_thresholding_classifier}
\end{definition}
Now we can state the main result of this paper:
\begin{theorem}
    For any CMM $M$ and \edit{stochastic} classifier $\hat Y$, there is an RTC $\hat Y_{p,t,\eta}$ with $M(\hat Y_{p,t,\eta}) \geq M(\hat Y)$. In particular, if $M$ is maximized by any stochastic classifier, then $M$ is maximized by a RTC.
    \label{thm:generalized_bayes}
\end{theorem}

As a special case of Theorem~\ref{thm:generalized_bayes}, the classical Bayes classifier corresponds to $M(C) = \truenegative + \truepositive$, $p = 0$, and $t = 0.5$. 
However, as discussed in the next paragraph, without stronger assumptions, Theorem~\ref{thm:generalized_bayes} does \emph{not} hold for deterministic classifiers.
Since RTCs generalize both the RTC structure and optimality properties of the Bayes classifier, we also refer to them as \emph{generalized Bayes classifiers}. 
We note that \emph{existence} of any maximizer $\hat Y$ of $M(C_{\hat Y})$ may depend on specific properties, such as (semi)continuity or convexity of $M$, which we do not investigate here.

We emphasize that Theorem~\ref{thm:generalized_bayes} makes absolutely no assumptions on the distribution of the data. In particular, all prior characterizations of optimal classifiers under general performance metrics assume that the distribution of the class probability $\eta(X)$ is absolutely continuous~\citep{narasimhan2014statistical,narasimhan2015consistent,koyejo2014consistent,yan2018binary,wang2019consistent}\footnote{Exceptions for the case of $F_1$ score are \citet[][Lemma 12]{zhao2013beyond} and \citet[][Theorem 1]{lipton2014optimal}.},
and \citet{wang2019consistent} claim that regularity assumptions on $\eta(X)$ such as absolute continuity ``seem to be unavoidable''. Our Theorem~\ref{thm:generalized_bayes} is the first result to omit such assumptions, and we specifically show that this comes at the cost of the optimal classifier possibly being non-deterministic for a single atom of $\eta(X)$. Figure~\ref{fig:threshold_plot} visually compares the stochastic thresholding classifier in Theorem~\ref{thm:generalized_bayes} to prior approaches.
\begin{figure}[hbtp]
    \centering
    \includegraphics[width=\linewidth,clip,trim={2in 0 2in 0}]{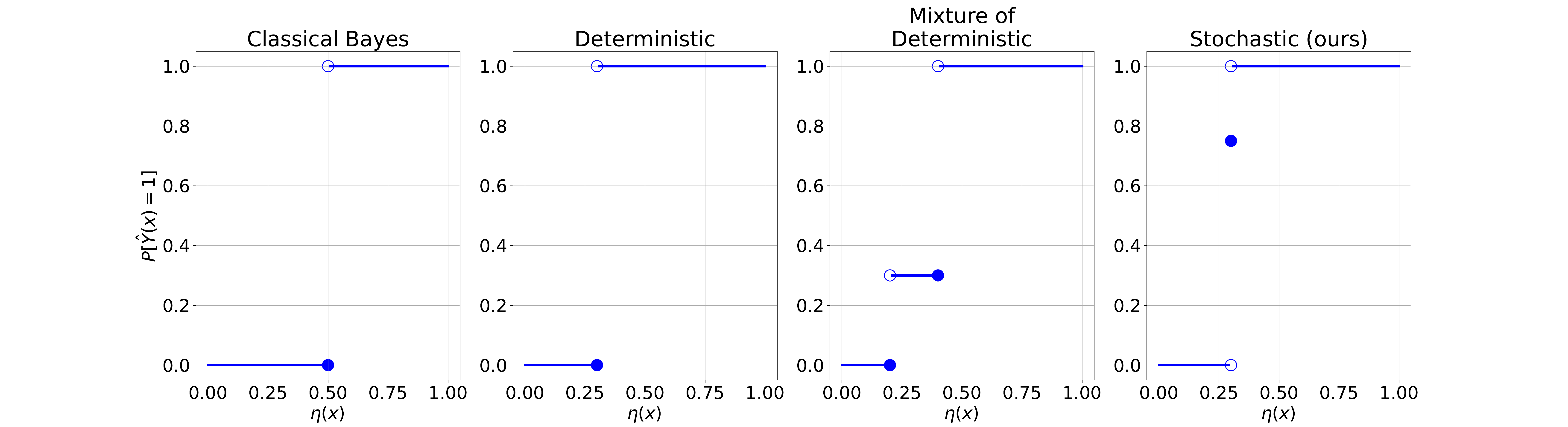}
    \caption{Examples of four different approaches to thresholding the regression function. Classical Bayes thresholding (Eq.~\eqref{eq:bayes_classifier}) always thresholds deterministically at $\eta(x) = 0.5$ to optimize accuracy. \citet{koyejo2014consistent,narasimhan2014statistical} and others have suggested using other Deterministic thresholds (e.g., $\eta(x) = 0.3$, shown here) to optimize other CMMs, assuming $\eta(X)$ is absolutely continuous. \citet{wang2019consistent} showed that the optimal classifier can always be written as a Mixture of Deterministic (MD) classifiers (e.g., a $(0.3, 0.7)$-mixture of thresholds at $\eta(x) = 0.2$ and $\eta(x) = 0.4$, shown here). Finally, we propose using a single Stochastic threshold (e.g., $(t, p) = (0.3, 0.75)$, shown here). Only MD and Stochastic approaches are optimal in general (for arbitrary CMMs, without $\eta(X)$ absolutely continuous), while Stochastic thresholding is strictly simpler than MD.}
    \label{fig:threshold_plot}
\end{figure}

The generality of Theorem~\ref{thm:generalized_bayes} necessitates a significantly more complex proof than prior work.
In particular, we prove Theorem~\ref{thm:generalized_bayes} in Appendix~\ref{app:generalized_bayes} using a series of variational arguments. 
Roughly speaking, given a classifier $\hat Y$, we construct a perturbation $\hat Y'$ of $\hat Y$ such that either $M(C_{\hat Y}) < M(C_{\hat Y'})$ or $\hat Y'$ is an RTC and $M(C_{\hat Y}) \leq M(C_{\hat Y'})$. 
Since, the classifier $\hat Y$ might be quite poorly behaved (e.g., its behavior on sets of $P_X$-measure $0$ could be arbitrary), the technical complexity lies in constructing admissible perturbations (i.e., those that are well-defined classifiers).
For this reason, the proof of Theorem~\ref{thm:generalized_bayes} involves a series of constructions of increasingly well-behaved classifiers.

Theorem~\ref{thm:generalized_bayes} tells us that a generalized Bayes classifier can always be written in terms of the regression function $\eta$ and two scalar parameters $(t, p)$ depending on the distribution of $\eta(X)$ and the CMM $M$. The next example shows that this characterization cannot be simplified without stronger assumptions:

\begin{example}
    Suppose $\X = \{0\}$ is a singleton, $\eta(0) \in (0, 1)$, and, for some $\theta > 0$, $M(C) = (\truepositive)^\theta \truenegative$. One can check that $M$ is a valid CMM. Suppose $\hat Y$ is an RTC. It is straightforward to compute that $M(C_{\hat Y}) = (p \eta(0))^\theta (1 - p) (1 - \eta(0)) 1\{t = \eta(0)\}$, and that
    $M(C_{\hat Y})$ is uniquely maximized by $p = \frac{\theta}{\theta + 1} \in (0, 1)$ and $t = \eta(0) \in (0, 1)$. This shows that both threshold parameters $p$ and $t$ in an RTC are necessary, in the absence of further assumptions on $M$ or $\eta$. This example also illustrates the need for stochasticity to optimize general CMMs. 
    Specifically, for any deterministic classifier $\hat Y$, either $\hat Y(0) = 0$ (so $\truepositive = 0$) or $\hat Y(0) = 1$ (so $\truenegative = 0$); in either case, $M(C_{\hat Y}) = 0$.
    \label{example:stochasticity_necessary}
\end{example}

This performance gap between stochastic and deterministic classifiers is closely related to Theorem 1 of \citet{cotter2019making}, which provides a closely related lower bound on how well a stochastic classifier can be approximated by a deterministic one, in terms of the probability assigned to atoms of $\eta(X)$. However, \citet{cotter2019making} only study how well stochastic classifiers can be approximated by deterministic ones (with the motivation of derandomizing classifiers), not whether stochastic classifiers can systematically outperform deterministic ones, as we show here.

\subsection{Relative Performance Guarantees in terms of the Generalized Bayes Classifier}
\label{subsec:relative_performance_guarantees}
Theorem~\ref{thm:generalized_bayes} motivates a two-step approach to imbalanced classification in which one first estimates the regression function $\eta$ and then selects a stochastic threshold $(t, p)$ that optimizes empirical performance $M(\hat C_{\hat Y})$. Such an approach has many practical advantages. For example, as we show in Appendix~\ref{app:computation}, a simple algorithm can exactly optimize the threshold $(t, p)$ over large datasets in $O(n \log n)$ time. Additionally, one can address covariate shift 
or retune a classifier trained under one CMM to perform well under another CMM, simply by re-optimizing $(t, p)$, which is statistically and computationally much easier than retraining a classifier from scratch. 
In this section, we focus on an advantage for theoretical analysis, namely that the error of such a classifier decomposes into errors in selecting $(t, p)$ and errors in estimating $\eta$, allowing the derivation of performance guarantees relative a generalized Bayes classifier.
All results in this section are proven in Appendix~\ref{app:relative_performance_guarantees}.

We first bound the performance difference of thresholding two regressors in terms of their $\L_\infty$ distance. 
This will allow us to bound error due to using a regressor $\hat\eta$ instead of the true $\eta$.
\begin{lemma}
    For $p,t \in [0, 1]$, $\eta, \eta' : \X \to [0, 1]$,
    $\left\| C_{\hat Y_{p,t,\eta}}\hspace{-3mm}- C_{\hat Y_{p,t,\eta'}} \right\|_\infty\hspace{-1mm}
      \leq \prob \left[ |\eta(X) - t| \leq \left\|\eta - \eta'\right\|_\infty \right]$.
    \label{lemma:approximation_error}
\end{lemma}

Intuitively, Lemma~\ref{lemma:approximation_error} bounds the largest difference in the confusion matrices of $\hat Y_{p,t,\eta}$ and $\hat Y_{p,t,\eta'}$ by the probability that the threshold $t$ lies between $\eta$ and $\eta'$. As we will show later, under a margin assumption, this can be bounded by the $\L_\infty$ distance between $\eta$ and $\eta'$.

Our next lemma bounds the worst-case error over thresholds $(t, p) \in [0, 1]$ of the empirical confusion matrix. 
This allows us to bound error due to using an empirical threshold $(\hat t, \hat p)$ instead of the threshold $(t^*, p^*)$ that is optimal for the true regression function.
\begin{lemma}
    Let $\eta : \X \to [0, 1]$ be any regression function. Then, with probability at least $1 - \delta$,
    \[\sup_{p, t \in [0, 1]} \left\| \hat C_{\hat Y_{p,t,\eta}} - C_{\hat Y_{p,t,\eta}} \right\|_\infty
        \leq \sqrt{\frac{8}{n} \log \frac{32(2n + 1)}{\delta}}.\]
    \label{lemma:estimation_error}
\end{lemma}
Lemma~\ref{lemma:estimation_error} follows from Vapnik-Chervonenkis (VC) bounds on the complexity of the set $\{\hat Y_{p,t,\eta} : p, t \in [0, 1]\}$ of possible RTCs with fixed regression function $\eta$. In fact, Appendix~\ref{app:relative_performance_guarantees} proves a more general bound on the error between empirical and true confusion matrices uniformly over any family $\F$ of stochastic classifiers in terms of the growth function of $\F$. Consequently, when $\F$ has finite VC dimension, we obtain uniform convergence at the fast rate $\sqrt{\log(n/\delta)/n}$. As we formalize later, this suggests that the difficulty in tuning an imbalanced classifier to optimize a CMM $M$ comes not from difficulty in estimating the confusion matrix but rather from the sensitivity of commonly used CMMs to the selected threshold.
Because Theorem~\ref{thm:generalized_bayes} shows that any CMM can be optimized by a RTC, we state here only the specific result for RTCs.

Before combining Lemmas~\ref{lemma:approximation_error} and \ref{lemma:estimation_error} to give the main result of this section, we note a margin assumption, which characterizes separation between the two classes:

\begin{definition}[Tsybakov Margin Condition]
    Let $C, \beta \geq 0$, $t \in (0, 1)$. A classification problem with covariate distribution $P_X$ and regression function $\eta$ satisfies a \emph{$(C, \beta)$-margin condition around $t$} if, for any $\epsilon > 0$, $\prob \left[ |\eta(X) - t| \leq \epsilon \right] \leq C \epsilon^\beta$.
\end{definition}

The Tsybakov margin condition, introduced by \citet{mammen1999smooth} for $t = 0.5$, is widely used to establish convergence rates for classification in terms of accuracy~\citep{audibert2007fast,arlot2011margin,chaudhuri2014rates}.
Together with the margin condition and a Lipschitz condition on the $M$, Lemmas~\ref{lemma:approximation_error} and \ref{lemma:estimation_error} give the following bound on sub-optimality of an RTC if the threshold is selected by maximizing $M$ over the empirical confusion matrix:

\begin{corollary}
    Let $\eta : \X \to [0, 1]$ be the true regression function and $\hat\eta : \X \to [0, 1]$ be any regressor.
    \begin{align*}
        \text{ Let } \quad
        \left( \hat p, \hat t \right) := \argmax_{(t, p) \in [0, 1]^2} M \left( \hat C_{\hat Y_{p,t,\hat\eta}} \right)
        \quad \text{and} \quad \left( p^*, t^* \right) := \argmax_{(t, p) \in [0, 1]^2} M \left(C_{\hat Y_{p,t,\eta}} \right)
    \end{align*}
    denote the empirical and true optimal thresholds, respectively. Suppose $M$ is Lipschitz continuous with constant $L_M$ with respect to the uniform ($\L_\infty$) metric on $\C$. Finally, suppose $P_X$ and $\eta$ satisfy a $(C, \beta)$-margin condition around $t^*$. Then, with probability $\geq 1 - \delta$,
    \begin{align*}
        M\left(C_{\hat Y_{p^*,t^*,\eta}}\right) - M\left(C_{\hat Y_{\hat p,\hat t,\hat\eta}}\right)
          \leq L_M \left( C\left\|\eta - \hat\eta\right\|_\infty^\beta + 2 \sqrt{\frac{8}{n} \log \frac{32(2n + 1)}{\delta}} \right).
    \end{align*}
    \label{corr:CMM_error_decomposition}
\end{corollary}

\section{Uniform Error of the \texorpdfstring{$k$}{k}NN Regressor}
\label{sec:uniform_error}
In the previous section, we bounded relative performance of an RTC in terms of uniform ($\L_\infty$) loss of the regression function estimate. 
Here, we bound uniform loss of one such regressor, the widely used $k$-nearest neighbor ($k$NN) regressor.
Our analyses include a parameter $r$, introduced in Section~\ref{subsec:uniform_class_imbalance}, that characterizes a novel sub-type of class imbalance, which we call Uniform Class Imbalance. 
This leads to insights about how the behavior of the $k$NN classifier depends not only on the degree, but also on the structure, of class imbalance in a given dataset. We begin with some notation:

\begin{definition}[$k$-Nearest Neighbor Regressor]
    Given a point \(x \in \X\), let $\sigma(x)$ denote a permutation of $\{1,...,n\}$ such that
    $X_{\sigma_i(x)}$ is the $i^{th}$-nearest neighbor of $x$ among $X_1,...,X_n$.
    For integers $k \in [1, n]$, the \emph{$k$NN regressor} $\hat\eta_k : \X \to [0, 1]$ is defined as
    \begin{align}
        \hat\eta_k(x) = \frac{1}{k} \sum_{i = 1}^k Y_{\sigma_i(x)},
        \quad \text{ for all } x \in \X.
        \label{eqn:KNNRegressor}
    \end{align}
\end{definition}

\label{subsec:uniform_class_imbalance}
We now formalize a novel sub-type of class imbalance:
\begin{definition}[Uniform Class Imbalance (UCI)]
Write the regression function as $\eta = r \zeta$, where $r \in (0, 1]$ and $\zeta : \X \to [0, 1]$ is a regression function with $\sup_{x \in \X} \zeta(x) = 1$. 
A classification problem has Uniform Class Imbalance (UCI) in the number of samples \(n\) if \(r \to 0\) as \(n \to \infty\).
\label{definition:uniformClassImbalance}
\end{definition}
Intuitively, in UCI, the class $Y = 1$ is rare regardless of $X$.
This includes ``difficult'' classification problems where the covariate $X$ provides only partial information about the class $Y$ and examples from the rare class lie deep within the distribution of the common class.
Examples include rare disease diagnosis~\citep{schaefer2020use} or fraud detection~\citep{awoyemi2017credit}. 
In practice, the classifier's role is often to flag ``high-risk'' samples $X$, those with $\eta(X)$ relatively high, for follow-up investigation.
UCI can be distinguished from ``easier'' problems in which, for some $x \in \X$, $\eta(X) \approx 1$ and so, given enough training data, a classifier can confidently assign the label $Y = 1$.
These include well-separated classes or deterministic problems (e.g., protein structure prediction; \citet{noe2020machine}).

To our knowledge, such notions of class imbalance have not previously been distinguished. 
In the particular case of data drawn from a logistic model, UCI reduces to the notion of class imbalance described in \citet{wang2020logistic}; however, UCI applies in more general contexts.

\subsection{Uniform Risk Bounds}
We now present bounds (proven in Appendix~\ref{app:knn_upper_bound_proofs}) on uniform error $\left\|\eta - \hat\eta\right\|_\infty = \sup_{x \in \X} |\eta(x) - \hat\eta(x)|$ of the $k$NN regressor $\hat\eta_k$. First, recall two standard quantities, covering numbers and shattering coefficients, by which we measure complexity of the feature space:

\begin{definition}[Covering Number]
Suppose $(\X, \rho)$ is a totally bounded metric space. Then, for any $\epsilon > 0$, the \emph{$\epsilon$-covering number} $N(\epsilon)$ of $(\X, \rho)$ is the smallest integer such that there exist $N(\epsilon)$ points $x_1,...,x_{N(\epsilon)} \in X$ satisfying
$\X \subseteq \bigcup_{i = 1}^{N(\epsilon)} B(x_i, \epsilon)$.
\end{definition}

\begin{definition}[Shattering Coefficient]
For integers $n > 0$,
the \emph{shattering coefficient of balls} in $(\X, \rho)$ is $\displaystyle S(n) = \sup_{x_1,...,x_n \in \X} \left| \left\{ \{x_1,...,x_n\} \cap B(x, \epsilon) : x \in \X, \epsilon \geq 0 \right\} \right|$.
\end{definition}

We now state two assumptions data distribution $P_{X, Y}$:

\begin{assumption}[Dense Covariates Assumption]
    For some $p_*,\epsilon^*,d > 0$, the marginal distribution $P_X$ of covariates is lower bounded, for any $x \in \X$ and $\epsilon \in (0,\epsilon^*]$, by $P_X(B_\epsilon(x)) \geq p_* \epsilon^d$.
    \label{assumption:denseCovariates}
\end{assumption}

Assumption~\ref{assumption:denseCovariates} ensures that each query point's nearest neighbors are sufficiently near to be informative. We also assume that the regression function $\zeta$ is smooth:

\begin{assumption}[H\"older Continuity]
    For some $\alpha \in (0, 1]$, $L := \sup_{x \neq x' \in \X} \frac{\left| \zeta(x) - \zeta(x') \right|}{\rho^\alpha(x, x')} < \infty$.
    \label{assumption:holderContinuity}
\end{assumption}
We now state our upper bound on uniform error:

\begin{theorem}
    Under Assumptions~\ref{assumption:denseCovariates} and \ref{assumption:holderContinuity}, whenever $k / n \leq p_*(\epsilon^*)^d / 2$,
    for any
    $\delta > 0$, with probability at least $1 - N\left( \left( 2k / (p_* n) \right)^{1/d} \right) e^{-k/4} - \delta$,
    we have
    \begin{align}
        \left\|\eta - \hat\eta\right\|_\infty
          \leq 2^\alpha Lr\left( \frac{2k}{p_* n} \right)^{\alpha/d}
          + \frac{2}{3k} \log \frac{2 S(n)}{\delta} + \sqrt{\frac{2r}{k} \log \frac{2 S(n)}{\delta}}.
        \label{eq:uniform_error_bound}
    \end{align}
    If $r \in O \left( (\log S(n)/n \right)$, this bound is minimized by $k \asymp n$, giving $\left\|\eta - \hat\eta\right\|_\infty \in O_P \left( (\log S(n))/n \right)$.
    Otherwise,
    this bound is minimized by $k \asymp n^{\frac{2\alpha}{2\alpha+d}} (\log S(n))^{\frac{d}{2\alpha+d}} r^{-\frac{d}{2\alpha + d}}$,
    giving
    \[\left\|\eta - \hat\eta\right\|_\infty \in O_P \left( \left( (\log S(n))/n \right)^\frac{\alpha}{2\alpha + d} r^\frac{\alpha + d}{2\alpha + d} \right).\]
    \label{thm:unif_convergence}
\end{theorem}

Of the three terms in \eqref{eq:uniform_error_bound}, the first term, of order $r(k/n)^{\alpha/d}$, comes from smoothing bias of the $k$NN classifier. 
The second and third terms are due to label noise, with the second term dominating under extreme class imbalance $r \in O \left(\log S(n) / n \right)$ and the third term dominating otherwise. 
Theorem~\ref{eq:uniform_error_bound} shows that, under UCI, one should use a much larger choice of the tuning parameter $k$ than in the case of balanced classes; indeed, setting $k \asymp n^{\frac{2\alpha}{2\alpha+d}} (\log S(n))^{\frac{d}{2\alpha+d}}$, which is optimal in the balanced case, gives a rate that is suboptimal by a factor of $r^{-d/(4\alpha + d)}$.

The following example demonstrates how to apply Theorem~\ref{thm:unif_convergence} in a concrete setting of interest:

\begin{corollary}[Euclidean, Absolutely Continuous Case]
Suppose $(\X,\rho) = ([0,1]^d,\|\cdot\|_2)$ is the unit cube in $\R^d$, equipped with the Euclidean metric, and $P_X$ has a density that is lower bounded away from $0$ on $\X$.
Then,
for $k \asymp n^{\frac{2\alpha}{2\alpha+d}} (\log n)^{\frac{d}{2\alpha+d}} r^{-\frac{d}{2\alpha + d}}$,
$\left\|\eta - \hat\eta\right\|_\infty
  \in O_P \left( \left( (\log n)/n \right)^{\frac{\alpha}{2\alpha+d}} r^\frac{\alpha + d}{2\alpha + d} \right)$.
  \label{corrollary:euclidean_AC_uniform_rate}
\end{corollary}
The most problematic term in this bound is the exponential dependence on the dimension $d$ of the covariates. Fortunately, since Theorem~\ref{thm:unif_convergence} utilizes covering numbers, it improves if the covariates exhibit structure, such as that of a low-dimensional manifold. We illustrate this in detail in Appendix~\ref{app:knn_upper_bound_proofs}.

We close with a minimax lower bound, proven in Appendix~\ref{app:knn_lower_bound_proofs}, on the uniform error of any estimator, over $(\alpha, L)$-H\"older regression functions. Up to a polylogarithmic factor in $r$, the rate of this lower bound matches that in Theorem~\ref{thm:unif_convergence}, suggesting that both bounds are quite tight.
\begin{theorem}
    Suppose $\X = [0,1]^d$ is the $d$-dimensional unit cube and $X \sim \operatorname{Uniform}(\X)$. Let $\Sigma^\alpha(L)$ denote the family of $(\alpha, L)$-H\"older continuous regression function.
    Then, there exist constants $n_0$ and $c > 0$ (depending only on $\alpha$, $L$, and $d$) such that, for all $n \geq n_0$ and any estimator $\hat \eta$,
    \[\sup_{\zeta \in \Sigma^\alpha(L)} \mathop{\prob}
    \left[ \left\|\eta - \hat \eta \right\|_\infty \geq c \left( \frac{\log(nr)}{n} \right)^{\frac{\alpha}{2\alpha + d}} r^\frac{\alpha + d}{2\alpha + d} \right] \geq \frac{1}{8}.\]
    \label{theorem:UniformErrorLowerBound}
\end{theorem}

\paragraph{Discussion}
Plugging the above upper bounds on $\left\|\eta - \hat \eta \right\|_\infty$ into Corollary~\ref{corr:CMM_error_decomposition} provides error bound under arbitrary CMMs, in terms of the sample size $n$, hyperparameter $k$, UCI degree $r$, and complexity parameters (margin $\beta$, smoothness $\alpha$, intrinsic dimension $d$, etc.) of $\X$ and $P_{X,Y}$.
Thus, these results collectively give some of the first complete finite-sample guarantees under general performance metrics used for imbalanced classification.
Our analysis shows that, under severe UCI, the optimal $k$ is much larger than in balanced classification, whereas this same $k$ leads to sub-optimal, or even inconsistent, estimates of the regression function under other (nonuniform) forms of class imbalance.

\section{Numerical Experiments}
We provide two numerical experiments to illustrate our results from Sections~\ref{sec:generalized_bayes} and \ref{sec:uniform_error}. 
We repeat each experiment $100$ times and present average results with $95\%$ confidence intervals computed using the central limit theorem.
Python implementations and instructions for reproducing each experiment can be found at \url{https://gitlab.tuebingen.mpg.de/shashank/imbalanced-binary-classification-experiments}. Further technical details regarding the experiments can be found in Appendix~\ref{app:further_experimental_details}, while Appendix~\ref{app:credit_card} explores some predictions of our theoretical results on real data from a credit card fraud detection problem.

\paragraph{Experiment 1}
Example~\ref{example:stochasticity_necessary} showed that, under general CMMs, deterministic RTCs are sometimes unable to approach optimal classification performance, necessitating stochastic RTCs. This experiment demonstrates this gap numerically.
Suppose $\X = [0, 1]$, over which $X$ is uniformly distributed, and for all $x \in \X$,
$\eta(x) = 0.5 \cdot 1\{1/3 \leq x < 2/3\} + 1\{2/3 \leq x\}$.

Consider the CMM $M(C) = \truepositive \cdot \truenegative$. Similar to the analysis in Example~\ref{example:stochasticity_necessary}, the optimal value of $M(C) = (5/12)^2 = \frac{25}{144}$ is achievable only by a stochastic classifier, whereas as deterministic classifiers achieve at most $M(C) = 1/6 < 25/144$.

For $10$ logarithmically spaced values of $n$ between $10^2$ and $10^4$, we drew $n$ independent samples of $(X, Y)$ according the above distribution. Using this training data, we selected optimal deterministic and stochastic thresholds $t \in [0, 1]$ and $(t, p) \in [0, 1]^2$ for the $k$NN classifier by maximizing $M(\hat C)$ over $10^4$ uniformly spaced values in $[0, 1]$ and $[0, 1]^2$, respectively. Since, in this example, $\alpha = d = 1$, we set $k = \lfloor n^{2/3} \rfloor$ as suggested by Theorem~\ref{corrollary:euclidean_AC_uniform_rate}. As another point of comparison, we also include a very different deterministic classifier, a random forest, trained with default parameters of Python's \texttt{scikit-learn} package. We estimated $M(C)$ using $1000$ more independently generated test samples of $(X, Y)$.
Figure~\ref{fig:experiment2} shows regret, i.e., sub-optimality of each classifier relative to the optimal classifier, in terms of $M(C)$.
Consistent with our analysis, regrets of the deterministic classifiers are bounded away from $0$, while regret of the stochastic classifier vanishes as $n$ increases.

\paragraph{Experiment 2}
This experiment demonstrates that making classifiers robust to severe class imbalance requires distinguishing different sub-types of class imbalance, such as UCI.
Suppose $\X = [0, 1]$, $X \sim \operatorname{Uniform}([0, 1])$, and $r \in (0, 1)$. 
Consider two regression functions $\eta_1(x) = r(1 - x)$ and $\eta_2(x) = \max\{0, 1 - x/r\}$. $\eta_1$ and $\eta_2$ exhibit the same overall class imbalance, with $r/2$ proportion of samples from class $1$. 
The regression function $\eta_1$ satisfies UCI of degree $r$, whereas $\eta_2$ does not satisfy a nontrivial degree of UCI.
For sufficiently small $r \in (0, 1)$, specifically $r \in o \left( n^{-d/(2\alpha + 2d)} \right)$, Theorem~\ref{thm:unif_convergence} gives that the optimal choice of $k$ under $\eta_1$ satisfies $k \in \omega(rn)$. 
On the other hand, if $k \in \omega(rn)$, then, under $\eta_2$, $\E \left[ \hat\eta_k(0) \right] \to 0$, so that $\hat\eta_k(0)$ is an inconsistent estimate of $\eta_2(0) = 1$.


\begin{figure}
    \centering
    \begin{subfigure}[b]{0.32\textwidth}
        \centering
        \includegraphics[width=\linewidth,trim={8mm 8mm 7mm 7mm},clip]{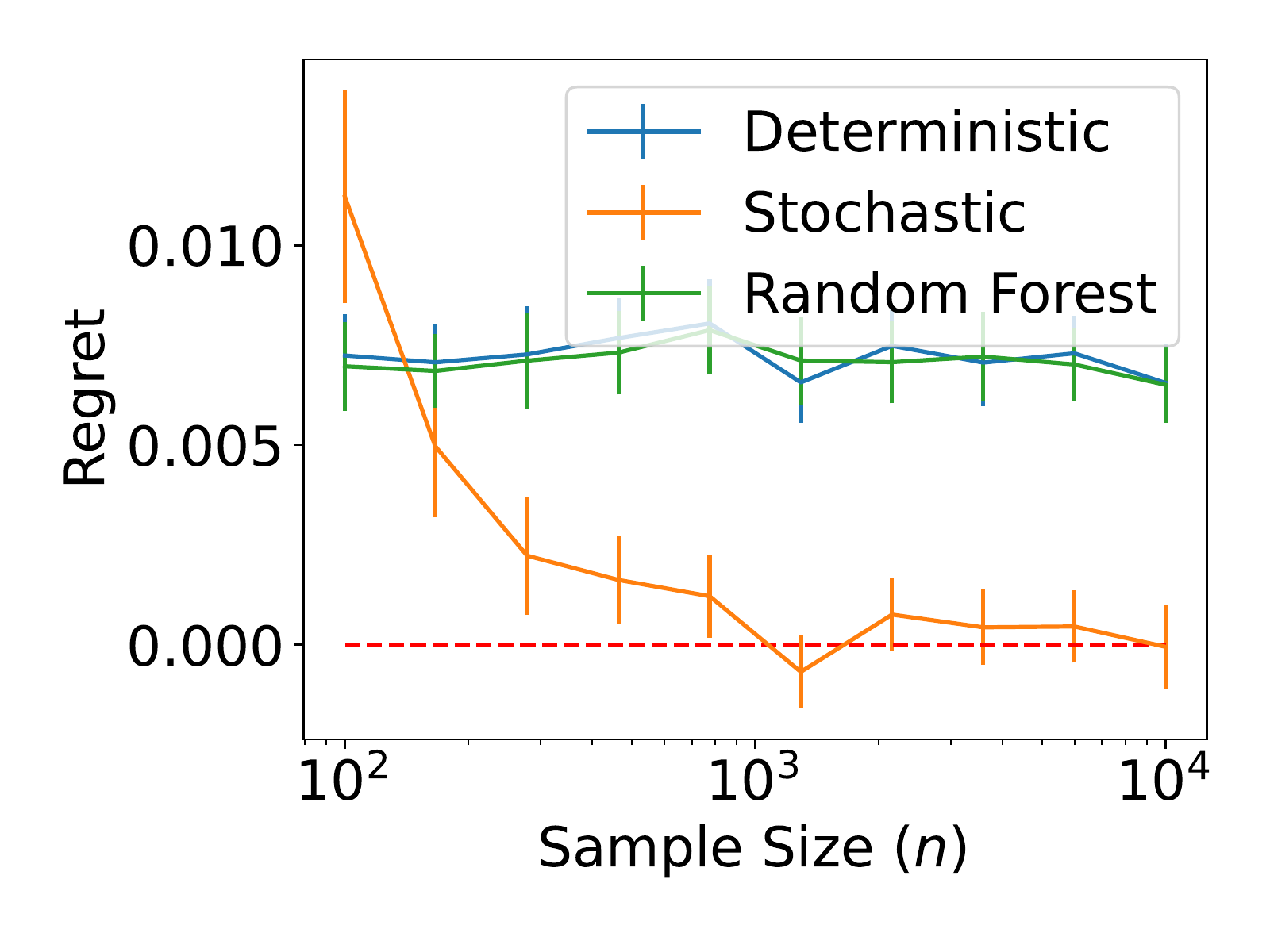}
        \caption{Experiment 1}
        \label{fig:experiment2}
    \end{subfigure}
    \hfill
    \begin{subfigure}[b]{0.65\textwidth}
        \centering
        \includegraphics[width=\linewidth,trim={35mm 0 42mm 0},clip]{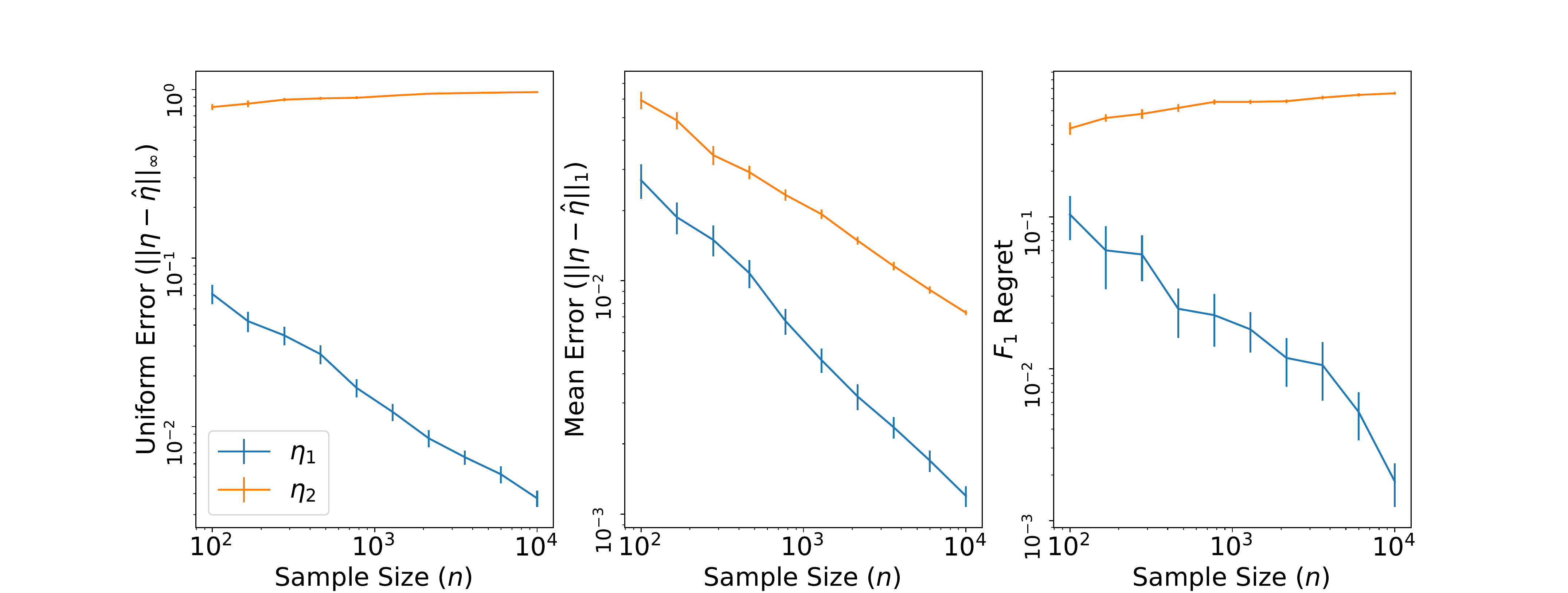}
        \caption{Experiment 2}
        \label{fig:experiment1}
    \end{subfigure}
    \caption{(a) Regret of deterministically and stochastically thresholding the $k$NN classifier, under the CMM $M(C) = \truepositive \cdot \truenegative$. (b) Uniform ($\L_\infty$) and average ($\L_1$) errors of $k$NN regressor, as well as $F_1$ regret of estimating regression functions $\eta_1$ and $\eta_2$. Error bars indicate $95\%$ confidence intervals.}
    \label{fig:experiments}
\end{figure}

For $10$ logarithmically spaced values of $n$ between $10^2$ and $10^4$, we drew $n$ independent samples of $(X, Y)$ according the joint distributions corresponding to each of $\eta_1$ and $\eta_2$. Since, in this example, $\alpha = d = 1$, to ensure $r \in o \left( n^{-d/(2\alpha + 2d)} \right)$, we set $r = n^{-1/2}$. As indicated by Corollary~\ref{corrollary:euclidean_AC_uniform_rate}, we set $k = \lfloor n^{2/3} r^{-1/3} \rfloor$. We then computed $\L_\infty$ and $\L_1$ distances between the $k$NN regressor (Eq.~\eqref{eqn:KNNRegressor}) and true regression function. We also drew $1000$ independent test samples of $(X, Y)$ and used these to estimate the $F_1$ score of thresholding the $k$NN regressor at a threshold $t$ determined by optimizing the empirical $F_1$ score (over the training data) over $100$ uniformly-spaced values of $t \in [0, 1]$.
Figure~\ref{fig:experiment1} shows the uniform ($\L_\infty$) error, the average ($\L_1$) error, and the $F_1$ regret,
which we bounded in Corollary~\ref{corr:CMM_error_decomposition}.
Consistent with our analysis above, the uniform ($\L_\infty$) error decays to $0$ under $\eta_1$ but not under $\eta_2$. Meanwhile, the average error ($\L_1$) decays to $0$ under both $\eta_1$ and $\eta_2$. Consistent with Corollary~\ref{corr:CMM_error_decomposition}, the $F_1$ regret of the thresholded classifier, which decays to $0$ under $\eta_1$ but not under $\eta_2$, mirrors performance of the regressor in uniform ($\L_\infty$) error rather than average ($\L_1$) error.

\section{Conclusions}
Our main conclusions are as follows.
First, without any assumptions on the data-generating distribution, the Bayes-optimal classifier generalizes from accuracy to other performance metrics using a stochastic thresholding procedure, while, in general, deterministic classifiers may not achieve Bayes-optimality.
This generalized Bayes classifier provides an optimal performance benchmark relative to which one can analyze classifiers that threshold estimates of the regression function.
This includes the $k$NN classifier, for which we provided new guarantees, including minimax-optimally under uniform loss in the presence of Uniform Class Imbalance. Our results imply that the parameter $k$ needs to be tuned differently for different sub-types of imbalanced classification, suggesting that developing reliable classifiers for severely imbalanced classes may require a more nuanced understanding of the data at hand.
Further work is needed to (a) understand how sub-types of class imbalance can be distinguished in practice, (b) develop adaptive classifiers that perform well under multiple imbalance sub-types, and (c) extend our results to the multiclass case.

While this paper focused on statistical properties of stochastic classification, we should point out that using stochastic classifiers in real applications may require careful consideration of possible downstream consequences. On one hand, Theorem~\ref{thm:generalized_bayes} provides justification for using (a limited degree of) stochasticity to break certain ties between classes: sometimes, this is provably necessary to optimize performance according to certain metrics. Stochastic classifiers can also be easier to train~\citep{cotter2019two,lu2020stochastic} or more robust to adversarial examples~\citep{pinot2022robustness}. However, stochastic classifiers have risks, including being harder to interpret, explain, or debug, and being vulnerable to manipulation by downstream users (e.g., a user might query the classifier multiple times to produce a desired prediction). Stochastic classifiers may also violate certain notions of fairness, as individuals with identical features might be assigned to different classes. Techniques for derandomizing classifiers~\citep{cotter2019making,wu2022metric} may help address these issues.

\begin{ack}
    This work was supported by the German Federal Ministry of Education and Research (BMBF) through the T\"ubingen AI Center (FKZ: 01IS18039B). 
\end{ack}

\bibliographystyle{plainnat}
\bibliography{refs.bib}

\input{checklist}
\input{appendix}

\end{document}

%% file: checklist.tex
\section*{Checklist}


\begin{enumerate}

\item For all authors...
\begin{enumerate}
  \item Do the main claims made in the abstract and introduction accurately reflect the paper's contributions and scope?
    \answerYes{}
  \item Did you describe the limitations of your work?
    \answerYes{}
  \item Did you discuss any potential negative societal impacts of your work?
    \answerNA{}
  \item Have you read the ethics review guidelines and ensured that your paper conforms to them?
    \answerYes{}
\end{enumerate}

\item If you are including theoretical results...
\begin{enumerate}
  \item Did you state the full set of assumptions of all theoretical results?
    \answerYes{}
        \item Did you include complete proofs of all theoretical results?
    \answerYes{See Appendices~\ref{app:generalized_bayes}, \ref{app:relative_performance_guarantees}, and \ref{app:UniformConvergenceProof}.}
\end{enumerate}

\item If you ran experiments...
\begin{enumerate}
  \item Did you include the code, data, and instructions needed to reproduce the main experimental results (either in the supplemental material or as a URL)?
    \answerYes{Code and instructions for reproducing the experimental results are included in the supplemental material.}
  \item Did you specify all the training details (e.g., data splits, hyperparameters, how they were chosen)?
    \answerYes{}
        \item Did you report error bars (e.g., with respect to the random seed after running experiments multiple times)?
    \answerYes{}
        \item Did you include the total amount of compute and the type of resources used (e.g., type of GPUs, internal cluster, or cloud provider)?
    \answerYes{See Appendix~\ref{app:further_experimental_details}.}
\end{enumerate}

\item If you are using existing assets (e.g., code, data, models) or curating/releasing new assets...
\begin{enumerate}
  \item If your work uses existing assets, did you cite the creators?
    \answerYes{}
  \item Did you mention the license of the assets?
    \answerYes{}
  \item Did you include any new assets either in the supplemental material or as a URL?
    \answerNA{}
  \item Did you discuss whether and how consent was obtained from people whose data you're using/curating?
    \answerNA{}
  \item Did you discuss whether the data you are using/curating contains personally identifiable information or offensive content?
    \answerNA{}
\end{enumerate}

\item If you used crowdsourcing or conducted research with human subjects...
\begin{enumerate}
  \item Did you include the full text of instructions given to participants and screenshots, if applicable?
    \answerNA{}
  \item Did you describe any potential participant risks, with links to Institutional Review Board (IRB) approvals, if applicable?
    \answerNA{}
  \item Did you include the estimated hourly wage paid to participants and the total amount spent on participant compensation?
    \answerNA{}
\end{enumerate}
\newpage

\end{enumerate}

%% file: appendix.tex
\appendix
\section{Derivation of the Generalized Bayes Classifier}
\label{app:generalized_bayes}

In this Appendix, we prove Theorem~\ref{thm:generalized_bayes}, in which we characterize a stochastic generalization of the Bayes classifier to arbitrary CMMs. We reiterate the theorem for the reader here:

\begin{customthm}{\ref{thm:generalized_bayes}}
    If $\argmax_{\hat Y \in \SC} M(C_{\hat Y}) \neq \varnothing$, then there exists a regression-thresholding classifier
    \[\hat Y_{p,t,\eta} \in \argmax_{\hat Y \in \SC} M(C_{\hat Y}).\]
\label{thm:generalized_bayes_app}
\end{customthm}

As described in the main paper, unlike prior results \citep{koyejo2014consistent,yan2018binary,wang2019multiclass}, we do not assume that the distribution of $\eta(X)$ is absolutely continuous. This makes proving Theorem considerably more complicated than these previous results.
We prove Theorem~\ref{thm:generalized_bayes} in a sequence of steps, constructing optimal classifiers in forms progressively closer to that of the generalized Bayes classifier described in Theorem~\ref{thm:generalized_bayes}. Specifically, we first show, in Lemma~\ref{lemma:regression_based_classifier}, that there exists an optimal classifier that is a (stochastic) function of the \edit{true} regression function $\eta\edit{^*}$. We then construct an optimal classifier in which this function of $\eta\edit{^*}$ is non-decreasing. Finally, we construct an optimal classifier in which this function of $\eta\edit{^*}$ is a threshold function, as in Theorem~\ref{thm:generalized_bayes}.

\begin{lemma}
    For any stochastic classifier $\hat Y : \X \to \B$, there is a stochastic classifier $\hat Y' : \X \to \B$ of the form
    \begin{equation}
        \hat Y'(x) \sim \text{Bernoulli}(f(\eta\edit{^*}(x))),
        \label{eq:regression_based_classifier}
    \end{equation}
    for some $f : [0, 1] \to [0, 1]$, such that $C_{\hat Y'} = C_{\hat Y}$.
    \label{lemma:regression_based_classifier}
\end{lemma}

\begin{proof}
    \edit{We start by defining the regression function and discussing formal probability notation.
    Let \(\mathcal{F} = \sigma(\eta^{*}(Z))\) be the \(\sigma\)-field generated by the true regression function \(\eta^{*}(Z)\).}
    Define $\hat Y' : \X \to \B$ by
    \[\hat Y'(x) \sim \text{Bernoulli} \left( \E_{Z \sim P_X} \left[ \hat Y(Z) \middle| 
    \mathcal{F}
    \right](\eta^{*}(x)) \right),
    \]
    \edit{where by the \(\eta^{*}(x)\) we are explicitly indicating that \(\eta^{*}(x)\) is the input to the conditional expectation, as the conditional expectation is a measurable function of \(\eta^{*}(Z)\).
    In the sequel, following standard conventions, we omit such notation.
    Note that $\hat Y'$ has the desired form and is defined almost surely.}
    
    \edit{Now, we get our final notes in place before proceeding with calculations.
    Let \(\mathcal{G} = \sigma(X)\) be the \(\sigma\)-field generated by \(X\).
    }

    \edit{A key step in the proof is showing the equality}
    \begin{align}
        \E\left[Y \hat Y'(X) \middle| \mathcal{F}\right]
        &= \eta^{*}(X) \E \left[ \hat Y (X) \middle| \mathcal{F}\right] 
        = \E\left[Y \hat Y(X) \middle| \mathcal{F}\right].
        \label{eqn:conditionalExpectation} 
    \end{align}
    \edit{We first provide an intuitive summary. Observe that conditioned on \(\eta^{*}(X)\), \(Y\) is independent of \(\hat Y\) and \(\hat{Y}'\) separately. Thus, we can integrate over \(Y\) and get the \(\eta^{*}(X)\). By the definition of \(\hat Y'\) we obtain the conditional expectation with respect to \(Y\), and we have the resulting equality.}
 
    \edit{For the time being, taking Equation~\eqref{eqn:conditionalExpectation} as fact, we complete the rest of the proof. Combining Equation~\eqref{eqn:conditionalExpectation} and the tower property of conditional expectation, we have} 
    \begin{align*}
    \truepositive_{\hat Y'}
         & = \E \left[ Y \hat Y'(X) \right]
         =  \E \left[ \E \left[ Y \hat Y'(X) \middle| \mathcal{F}  \right] \right] 
         =  \E \left[ \E \left[ Y \hat Y(X) \middle| \mathcal{F}  \right] \right] 
          = \E \left[ Y \hat Y(X) \right] 
          = \truepositive_{\hat Y}.
    \end{align*}
    Similarly, one can check that $C_{\hat Y'} = C_{\hat Y}$.
    \end{proof}
    
    \begin{proof}[Proof of Equation~\eqref{eqn:conditionalExpectation}]
    
    \edit{We now present the formal details to establish Equation~\eqref{eqn:conditionalExpectation}, which may be skipped if one is uninterested in the measure-theoretic details. We first to set up the random variables needed to formalize the above intuition. Let \(S\), \(T\), \(U\), \(V\), and \(W\) be random variables. Let \((S, T) = (\eta^{*}(X), X)\). We also require \(U\) and \(V\) to be uniform on \([0, 1]\). We need not specify the precise distribution \(W\). Further,  let \((S, T)\), \(U\), \(V\), and \(W\) all be independent. For notation purposes, it is convenient for the underlying probability space \((\Omega, \mathcal{G}, \prob)\) to be the product probability space of some probability spaces to the extent possible. In particular, we require the set \(\Omega =  \bigtimes_{i = S}^{W} \Omega_{i}\) and a set \(B\) in \(\mathcal{G}\) to have a product structure \(B =  \bigtimes_{i = S}^{W} B_{i}\) for ease of exposition. Note that the measure $\prob$ is not a product measure, as \(S\) and \(T\) are dependent, although it is a product measure with respect to the laws of \((S, T)\), \(U\), \(V\), \(W\).}
    
    \edit{Now, we define \(Y\), \(\hat Y'\), and \(\hat Y\) in terms of \(U\), \(V\), \(W\), and \(X\). Define
    \begin{align*}
        Y(\eta^{*}(X)) &= \ind\{U \leq \eta^{*}(X)\} \\
        \hat Y' (\eta^{*}(X)) &= \ind\{V \leq \E[\hat Y | \sigmafield]\} \\
        \hat Y(X) &= r(S, T, W),
    \end{align*}
    where here \(U\) serves as the noise in \(Y\), \(V\) serves as the internal randomization of \(\hat Y'\), and \(W\) serves as the possible internal randomization of \(\hat Y\). Here, \(\ind\) is the indicator taking the value 1 if the event in curly braces occurs and 0 otherwise, and  \(r\) is some function. Note that this is still of the desired form.}
    
    \edit{To prove the first part of Equation~\eqref{eqn:conditionalExpectation}, i.e., to show that the middle term \(M = \eta^{*}(X) \E [ \hat Y (X)| \mathcal{F}] = S \E[\hat Y | \sigmafield]\) is indeed the conditional expectation of the left of Equation~\eqref{eqn:conditionalExpectation}, we must verify two conditions: (i) \(M\) is \(\sigmafield\)-measurable and (ii) the integrals of \(Y \hat Y'\) and \(M\) are identical on any event \(A\) in \(\sigmafield\) \citep[page 221]{durrett2010probability}. Condition (i) is immediate, since \(M\) is a function of \(\eta^{*}(X)\) and nothing more.}
    
    \edit{For condition (ii), we do a bit of computation. Let \(A = A_{S} \times A_{T} \times A_{U} \times A_{V} \times A_{W}\) be an event in \(\sigmafield\).
    Note that because \(A\) is \(\sigmafield\)-measurable, the sets \(A_{U}\) and \(A_{V}\) must be all of \(\Omega_{U}\) and \(\Omega_{V}\) or the empty set. Suppose for the moment that both are non-empty. Then, we have
    \begin{align*}
    \int_{A} Y \hat Y' d\prob 
    &= \int_{A_{S}} \int_{A_{V}} \int_{A_{U}} \ind\left\{U \leq S\right\} \ind\{V \leq \E[\hat Y | \sigmafield]\} d\prob_{U} d\prob_{V} d\prob_{S} \\
    &= \int_{A_{S}} \int_{A_{V}} S \ind\{V \leq \E[\hat Y | \sigmafield]\} d\prob_{V} d\prob_{S} \\
    &= \int_{A_{S}} S \E[\hat Y | \sigmafield] d\prob_{S} \\
    &= \int_{A} S \E[\hat Y | \sigmafield] d\prob.
    \end{align*}
    Note that the crucial first and last equalities in which we break and reform the integral over the entire event in terms of its coordinates is possible due to Fubini's theorem for integrable functions. This proves the desired equality when \(A_{U} = \Omega_{U}\) and \(A_{V} = \Omega_{V}\), and from the preceding calculation we can see that the integrals are both 0 when either \(A_{U} = \emptyset\) or \(A_{V} = \emptyset\), and so the desired equality holds. This establishes the left equality of Equation~\eqref{eqn:conditionalExpectation}.}
    
    \edit{Establishing the right hand side of Equation~\eqref{eqn:conditionalExpectation} also takes a bit of calculation. First, we have already verified the measurability condition (i), and so all that remains is to check the integral equality. Again, let \(A\) be an event in \(\sigmafield\), and assume that \(A_{T}\), \(A_{U}\), and \(A_{W}\) are non-empty. We have
    \begin{align*}
    \int_{A} Y \hat Y d\prob
    &= \int_{A_{S}} \int_{A_{T}} \int_{A_{W}}  \int_{A_{U}} \ind\left\{U \leq S\right\} r(S, T, W) d\prob_{U} d\prob_{W} d\prob_{T} d\prob_{S} \\
    &= \int_{A_{S}} \int_{A_{T}} \int_{A_{W}} S \; r(S, T, W) d\prob_{W} d\prob_{T} d\prob_{S} \\
    &= \int_{A_{S}}  S\left(\int_{A_{T}} \int_{A_{W}} r(S, T, W) d\prob_{W} d\prob_{T}\right) d\prob_{S}.
    \end{align*}
    Note that Fubini's theorem is again used in the first step. In the event that any of \(A_{T}\), \(A_{U}\), or \(A_{W}\) is empty, then the above equation is 0 and equality holds. Now, we have to show that the random variable \(M' = \int_{A_{T}} \int_{A_{W}} r(S, T, W) d\prob_{W} d\prob_{T|S}\) is the conditional expectation of \(\hat Y\) with respect to \(\sigmafield\). Measurability is readily apparent, as \(M'\) is a function of \(S\) and no more. Now, we check the condition that \(\hat Y\) and \(M'\) have the same integral on an event \(A\) in \(\sigmafield\). Again assume that \(A_{U}\) and \(A_{W}\) are non-empty, observing in the calculation to follow that equality holds with the value 0 if either is empty. Using Fubini's theorem and some direct computation, we have
    \begin{align*}
    \int_{A} \hat Y d\prob
    &= \int_{A} r(S, T, W) d\prob \\
    &= \int_{A_{S}} \int_{A_{T}} \int_{A_{W}}   r(S, T, W)  d\prob_{W} d\prob_{T} d\prob_{S} \\
    &= \int_{A} \left(\int_{A_{T}} \int_{A_{W}} r(S, T, W) d\prob_{W} d\prob_{T}\right) d\prob \\
    &= \int_{A} M' d\prob.
    \end{align*}
    This completes the proof that the conditional expectation of \(\hat Y\) is \(M'\), and so it proves that the conditional expectation of \(Y \hat Y\) is \(\eta^{*}(X) \E[\hat Y | \sigmafield].\) This is the right equality of Equation~\eqref{eqn:conditionalExpectation}, thus completing the proof.}
\end{proof}


It follows from Lemma~\ref{lemma:regression_based_classifier} that, if $M(C_{\hat Y})$ is maximized by any \edit{stochastic} classifier, then it is maximized by a classifier $\hat Y$ of the form in Eq.~\eqref{eq:regression_based_classifier}. It remains to show that $f$ in Eq.~\eqref{eq:regression_based_classifier} can be of the form $z \mapsto p 1\{z = t\} + 1\{z > t\}$ for some threshold $(p, t) \in [0, 1]^2$. Before proving this, we give a simplifying lemma showing that the problem of maximizing a CMM can be equivalently framed as a particular functional optimization problem. This will allow us to to significantly simplify the notation in the subsequent proofs.

\begin{lemma}
    Let $M$ be a CMM, and suppose that $M(C_{\hat Y})$ is maximized (over $\SC$) by a classifier $\hat Y$ of the form
    \[\hat Y(x) \sim \operatorname{Bernoulli}(f(\eta(x))),\]
    for some $f^* : [0, 1] \to [0, 1]$. Let $f$ be a solution to the optimization problem
    \begin{equation}
    \max_{f : [0, 1] \to [0, 1]} \E[\eta(X) f(\eta(X))]
      \; \text{ s.t. } \;
      \E[(1 - \eta(X)) f(\eta(X))] \leq \E[(1 - \eta(X)) f^*(\eta(X))].
        \label{specific_opt}
    \end{equation}
    Then, the classifier
    \[\hat Y'(x) \sim \operatorname{Bernoulli}(f(\eta(x))),\]
    also maximizes $M(C_{\hat Y'})$ (over $\SC$).
    \label{lemma:CMM_neyman_pearson_equivalence}
\end{lemma}

\begin{proof}
    This result follows from the definition (Definition~\ref{def:CMM}) of a CMM. Specifically, by construction of $\hat Y'$,
    \[\truepositive_{\hat Y'} = \E[\eta(X) f(\eta(X))] \geq \E[\eta(X) f^*(\eta(X))] = \truepositive_{\hat Y}\]
    and
    \[\falsepositive_{\hat Y'} = \E[(1 - \eta(X)) f(\eta(X))] \leq \E[(1 - \eta(X)) f^*(\eta(X))] = \falsepositive_{\hat Y}.\]
    Moreover, \edit{since the proportions of positive and negative true labels are independent of the chosen classifier} (i.e., $\truepositive_{\hat Y'} + \falsenegative_{\hat Y'} = \truepositive_{\hat Y} + \falsenegative_{\hat Y}$ and $\falsepositive_{\hat Y'} + \truenegative_{\hat Y'} = \falsepositive_{\hat Y} + \truenegative_{\hat Y}$), we have
    \[C_{\hat Y'} =
    \begin{bmatrix}
        \truenegative_{\hat Y} + \epsilon_1 & \falsepositive_{\hat Y} - \epsilon_1 \\
        \falsenegative_{\hat Y} - \epsilon_2 & \truepositive_{\hat Y} + \epsilon_2,
    \end{bmatrix},\]
    where $\epsilon_1 := \falsepositive_{\hat Y} - \falsepositive_{\hat Y'} \in [0, \falsepositive_{\hat Y}]$ and $\epsilon_2 := \truepositive_{\hat Y'} - \truepositive_{\hat Y} \in [0, \falsenegative]$.
    Thus, by the definition (Definition~\ref{def:CMM}) of a CMM, $M(C_{\hat Y'}) \geq M(C_{\hat Y})$.
\end{proof}

Lemma~\ref{lemma:CMM_neyman_pearson_equivalence} essentially shows that maximizing any CMM $M$ is equivalent to performing Neyman-Pearson classification, at some particular false positive level $\alpha$ depending on $M$ (through $f^*$) and on the distribution of $\eta(X)$. For our purposes, this simplifies the remaining steps in proving Theorem~\ref{thm:generalized_bayes} by allowing us to ignore the details of the particular CMM $M$ and regression function $\eta$ and focus on characterizing solutions to an optimization problem of the form~\eqref{specific_opt} (see, specifically, \eqref{opt} below).

To characterize solutions to this optimization problem, we will utilize the following two measure-theoretic technical lemmas:

\begin{lemma}
    Let $\mu$ be a measure on $[0, 1]$ with $\mu([0, 1]) > 0$. Then, there exists $z \in \R$ such that, for all $\epsilon > 0$, $\mu([0, 1] \cap (z - \epsilon, z)) > 0$.
    \label{lemma:positive_measure_interval}
\end{lemma}

\begin{proof}
    We prove the contrapositive. Suppose that, for every $z \in [0, 1]$, there exists $\epsilon_z > 0$ such that $\mu([0, 1] \cap (z - \epsilon_z, z)) = 0$. The family $\mathcal{S} := \{[0, 1] \cap (z - \epsilon_z, z) : z \in \R \}$ is an open cover of $[0, 1]$. Since $[0, 1]$ is compact, there exists a finite sub-cover $\mathcal{S}' \subseteq \mathcal{S}$ of $[0, 1]$. Thus, by countable subaddivity of measures,
    \[\mu([0, 1]) \leq \sum_{S \in \mathcal{S}'} \mu(\mathcal{S}) = 0.\]
\end{proof}

\edit{\begin{lemma}
    Let $(\X, \Sigma, \mu)$ be a measure space, let $E, F \in \Sigma$ be measurable sets, and let $f : \X \to \R$ be a $\Sigma$-measurable function. If
    \[\essup_\mu f(E) > \esinf_\mu f(F),\]
    then there exist measurable sets $A \subseteq E$ and $B \subseteq F$ with $\mu(A), \mu(B) > 0$, and
    \[\inf_{x \in A} f(x) > \sup_{x \in B} f(x).\]
    \label{lemma:inf_sup}
\end{lemma}}

\edit{\begin{proof}
    If $\essup_\mu f(E) > \esinf_\mu f(F)$,
    then there exist $a > b$ such that
    $\essup_\mu f(E) > a > b > \esinf_\mu f(F)$.
    Since $\essup_\mu f(E) > a$,
    it follows that
    $P_Z(E \cap \{z : f(z) \geq a\}) > 0$.
    Similarly, since $\esinf_\mu f(F) < b$, it follows that
    $P_Z(F \cap \{z : f(z) \leq b\}) > 0$.
    Hence, letting $A = E \cap \{z : f(z) \geq a\}$ and $B = F \cap \{z : f(z) \leq b\}$, we have
    \[\inf_{z \in A} f(z) \geq a > b \geq \sup_{z \in B} f(z).\]
\end{proof}}

We are now ready for the main remaining step in the proof of Theorem~\ref{thm:generalized_bayes}, namely characterizing solutions of (a generalization of) the optimization problem~\eqref{specific_opt}:

\begin{lemma}
    Let $Z$ be a $[0, 1]$-valued random variable, and let $c \in [0, 1]$. Suppose that the optimization problem
    \begin{equation}
        \max_{f : [0, 1] \to [0, 1] \text{ measurable}} \E[Z f(Z)]
        \quad \text{ subject to } \quad \E[(1 - Z) f(Z)] \leq c
        \label{opt}
    \end{equation}
    has a solution. Then, there is a solution to \eqref{opt} that is a stochastic threshold function.
    \label{lemma:opt_problem}
\end{lemma}

\begin{proof}
    Suppose that there exists a solution $f$ to \eqref{opt}. We will construct a stochastic threshold function that solves \eqref{opt} in two main steps. First, we will construct a monotone solution to \eqref{opt}. Second, we will show that this monotone solution is equal to a stochastic threshold function except perhaps on a set of probability $0$ with respect to $Z$. This stochastic threshold function is therefore a solution to \eqref{opt}.
    
    \textbf{Construction of Monotone Solution to \eqref{opt}:}
    Define
    \[g(z) := \essup_{P_Z} f([0, z])
      \quad \text{ and } \quad
      h(z) := \esinf_{P_Z} f((z, 1]),\]
    where the essential supremum and infimum are taken with respect to the measure $P_Z$ of $Z$, with the conventions $g(z) = 0$ whenever $P_Z([0, z]) = 0$ and $h(z) = 1$ whenever $P_Z((z, 1]) = 0$. We first show that, for all $z \in [0, 1]$, $g(z) \leq h(z)$. We will then use this to show that $g = f$ except on a set of $P_Z$ measure $0$ (i.e., $P_Z(\{z \in [0, 1] : g(z) \neq f(z)\}) = 0$). Therefore, both $\E[Z g(Z)] = \E[Z f(Z)]$ and $\E[(1 - Z) g(Z)] = \E[(1 - Z) f(Z)]$. Since $g : [0, 1] \to [0, 1]$ is clearly monotone non-decreasing, the result follows.
    
    Suppose, for sake of contradiction, that, for some $z \in [0, 1]$, $g(z) > h(z)$. \edit{By Lemma~\ref{lemma:inf_sup},} there exist $A \subseteq [0, z]$ and $B \subseteq (z, 1]$ such that $\inf_{z \in A} f(z) > \sup_{z \in B} f(z)$ and $P_Z(A), P_Z(B) > 0$.
    Define $z_A := \E[Z|Z \in A]$ and $z_B := \E[Z|Z \in B]$, and note that, since $A \subseteq [0, z]$ and $B \subseteq (z, 1]$, $z_A < z_B$. Define,
    \begin{align*}
        \epsilon := \min \left\{ \frac{P_Z(A) (1 - z_A)}{P_Z(B) (1 - z_B)} \inf_{z \in A} f(z),\quad 1 - \sup_{z \in B} f(z) \right\} > 0
    \end{align*}
    and define $\phi : [0, 1] \to [0, 1]$ by
    \[\phi(z)
      := \left\{
          \begin{array}{cc}
               f(z) - \epsilon \frac{P_Z(B) (1 - z_B)}{P_Z(A) (1 - z_A)} & \text{ if } z \in A \\
               f(z) + \epsilon & \text{ if } z \in B \\
               f(z) & \text{ otherwise,}
          \end{array}
      \right.\]
    noting that, by construction of $\epsilon$, $\phi(z) \in [0, 1]$ for all $z \in [0, 1]$.
    Then, by construction of $\phi$,
    \begin{align*}
        \E[(1 - Z) \phi(Z)] - \E[(1 - Z) f(Z)]
        & = - (1 - z_A) \epsilon \frac{P_Z(B) (1 - z_B)}{P_Z(A) (1 - z_A)} P_Z(A) + (1 - z_B) \epsilon P_Z(B) \\
        & = 0 \cdot \epsilon = 0,
    \end{align*}
    while
    \begin{align*}
        \E[Z \phi(Z)] - \E[Z f(Z)]
        & = - z_A \epsilon \frac{P_Z(B) (1 - z_B)}{P_Z(A) (1 - z_A)} P_Z(A) + z_B \epsilon P_Z(B) \\
        & = \left( - \frac{z_A}{1 - z_A}(1 - z_B) + z_B \right) \epsilon P_Z(B)
          > 0,
    \end{align*}
    since the function $z \mapsto \frac{z}{1 - z}$ is strictly increasing. This contradicts the assumption that $f$ optimizes~\eqref{opt}, implying $g \leq h$.
    
    We now show that $g = f$ except on a set of $P_Z$ measure $0$. First, note that, if $g(z) \neq f(z)$, then $g(z) = \essup f([0, z]) = \essup f([0, z))$, and so $g$ is left-continuous at $z$.
    
    For any $\delta > 0$, define
    \[A_\delta := \left\{ z \in [0, 1] : g(z) < f(z) - \delta \right\}
        \quad \text{ and } \quad
        B_\delta := \left\{ z \in [0, 1] : g(z) > f(z) + \delta \right\}.\]
    Since
    \[\{z \in [0, 1] : g(z) < f(z)\} = \bigcup_{j = 1}^\infty \left\{ z \in [0, 1] : g(z) < f(z) - \frac{1}{j} \right\}\]
    and 
    \[\{z \in [0, 1] : g(z) > f(z)\} = \bigcup_{j = 1}^\infty \left\{ z \in [0, 1] : g(z) > f(z) + \frac{1}{j} \right\},\]
    by countable subadditivity, it suffices to show that $P_Z(A_\delta) = P_Z(B_\delta) = 0$ for all $\delta > 0$.
    
    Suppose, for sake of contradiction, that $P_Z(A_\delta) > 0$. Applying Lemma~\ref{lemma:positive_measure_interval} to the measure $E \mapsto P_Z(A_\delta \cap E)$, there exists $z \in \R$ such that, for any $\epsilon > 0$, $P_Z(A_\delta \cap (z - \epsilon, z)) > 0$. Since $g$ is continuous at $z$, there exists $\epsilon > 0$ such that $g(z - \epsilon) \geq g(z) - \delta$, so that, for all $z \in A_\delta \cap (z - \epsilon, z)$, $f(z) > g(z) + \delta$. Then, since $P_Z(A_\delta \cap (z - \epsilon, z)) > 0$, we have the contradiction
    \[g(z) \geq \essup f(A_\delta \cap (z - \epsilon, z)) > g(z).\]
    
    On the other hand, suppose, for sake of contradiction, that $P_Z(B_\delta) > 0$. Applying Lemma~\ref{lemma:positive_measure_interval} to the measure $E \mapsto P_Z(B_\delta \cap E)$, there exists $z \in \R$ such that, for any $\epsilon > 0$, $P_Z(B_\delta \cap (z - \epsilon, z)) > 0$. Since $g$ is continuous at $z$, there exists $\epsilon > 0$ such that $g(z - \epsilon) \geq g(z) - \delta$. At the same time, since $g$ is non-decreasing, for $t \in B_\delta \cap (z - \epsilon, z)$, $f(t) < g(t) - \delta \leq g(z) - \delta$. Thus, since $P_Z(B_\delta \cap (z - \epsilon, z)) > 0$, we have $h(z - \epsilon) < g(z) - \delta < g(z - \epsilon)$, contradicting the previously shown fact that $g \leq h$.
    
    To conclude, we have shown that $P_Z(\{z \in [0, 1] : g(z) \neq f(z)\}) = 0$.

    \textbf{Construction of a Stochastic Threshold Solution:} We now construct a solution to \eqref{opt} that is equal to a stochastic threshold function (i.e., a function that has the form $p 1\{z = t\} + 1\{z > t\}$) except on a set of $P_Z$-measure $0$. To show this, it suffices to construct a function $f : [0, 1] \to [0, 1]$ such that (a) $f$ is monotone non-decreasing and (b) the set $f\inv((0, 1))$ is the union of the singleton $\{t\}$ and a set of $P_Z$-measure $0$.
    
    From the previous step of this proof, we may assume that we have a solution $f$ to~\eqref{opt} that is monotone non-decreasing. It suffices therefore to show that $A := f\inv((0, 1))$ is the union of a singleton and a set of $P_Z$-measure $0$. Define
    \[t_0 := \inf \{z \in [0, 1] : P_Z(A \cap [0, z]) > 0\}
      \quad \text{ and } \quad
      t_1 := \sup\{z \in [0, 1] : P_Z(A \cap [z, 1]) > 0\}.\]
    Then, for all $\epsilon > 0$, $P_Z(A \cap [0, t_0 - \epsilon]) = P_Z(A \cap [t_1 + \epsilon, 1]) = 0$. Hence, if $t_0 = t_1$, then, since
    \[A \backslash \{t_0\} = \bigcup_{j = 1}^\infty A \cap \left( [0, t_0 - 1/j] \cup [t_1 + 1/j, 1] \right)\]
    by countable subadditivity, $P_Z(A \backslash \{t_0\}) = 0$, which implies that $A = \{t_0\} \cup (A \backslash \{t_0\})$ is the union of a singleton and a set of measure $0$.
    
    It suffices therefore to prove that $t_0 = t_1$. It is easy to see, from the definitions of $t_0$ and $t_1$, that $t_0 \leq t_1$. Suppose, for sake of contradiction, that $t_0 < t_1$. Then, there exists $t \in (t_0, t_1)$, and, by definition of $t_0$ and $t_1$, both $P_Z(A \cap [0, t)) > 0$ and $P_Z(A \cap (t, 1]) > 0$. For any $\delta \geq 0$, define
    \[B_\delta := \{z \in [0, t): \delta < f(z) < 1 - \delta\}
      \quad \text{ and } \quad C_\delta := \{z \in (t, 1]: \delta < f(z) < 1 - \delta\},\]
    so that $P_Z(B_0) > 0$ and $P_Z(C_0) > 0$. By countable subadditivity, there exists $\delta > 0$ such that $P_Z(B_\delta) > 0$ and $P_Z(C_\delta) > 0$.
    
    Define $\epsilon := \delta \cdot \min \{P_Z(B_\delta), P_Z(C_\delta)\}\} > 0$.
    Define $g : [0, 1] \to \R$ for all $z \in [0, 1]$ by
    \[g(z)
      = \left\{
          \begin{array}{cc}
            f(z) - \frac{\epsilon}{P_Z(B_\delta)} & \text{ if } z \in B_\delta \\
            f(z) + \frac{\epsilon}{P_Z(C_\delta)} & \text{ if } z \in C_\delta \\
            f(z)                                  & \text{ otherwise}.
          \end{array}
      \right.,\]
    and note that, by definition of $\epsilon$, $B_\delta$, and $C_\delta$, $g : [0, 1] \to [0, 1]$. Then,
    \[\E[g(Z)] - \E[f(Z)]
      = -\frac{\epsilon}{P_Z(B_\delta)} P_Z(B_\delta) + \frac{\epsilon}{P_Z(C_\delta)} P_Z(C_\delta)
      = 0,\]
    while
    \begin{align*}
        \E[Z g(Z)] - \E[Z f(Z)]
        & = -\E[Z | Z \in B_\delta] \frac{\epsilon}{P_Z(B_\delta)} P_Z(B_\delta)
          + \E[Z | Z \in C_\delta] \frac{\epsilon}{P_Z(C_\delta)} P_Z(C_\delta) \\
        & = \epsilon \left( \E[Z | Z \in C_\delta] - \E[Z | Z \in B_\delta] \right).
    \end{align*}
    Since $B_\delta \subseteq [0, t)$ and $C_\delta \subseteq (t, 1]$, this difference is strictly positive, contradicting the assumption that $f$ optimizes~\eqref{opt}.
\end{proof}

Combining Lemma~\ref{lemma:opt_problem} with Lemma~\ref{lemma:CMM_neyman_pearson_equivalence} completes the proof of our main result, Theorem~\ref{thm:generalized_bayes}.

\subsection{Extension to AUROC}
\label{app:AUROC}

For any regression function $\eta$, the receiver operating characteristic (ROC) function $\operatorname{ROC}_\eta : [0, 1] \to [0, 1]$ is
\begin{equation}
    \operatorname{ROC}_\eta(x) := \sup_{(p, t) \in [0, 1]^2} \truepositive_{\hat Y_{p,t,\eta}} \cdot 1\{\falsepositive_{\hat Y_{p,t,\eta}} \leq x\}
    \quad \text{ for all } \quad x \in [0, 1],
    \label{eq:ROC}
\end{equation}
i.e., $\operatorname{ROC}(x)$ is the maximum true positive \edit{probability} (over all regression-thresholding classifiers with regression function $\eta$) achievable while keeping the false positive \edit{probability} below $x$. The area under the ROC curve (AUROC) is then given by
\begin{equation}
    \operatorname{AUROC}_\eta = \int_0^1 \operatorname{ROC}_\eta(x) \, dx.
    \label{eq:AUROC}
\end{equation}
While AUROC is not a CMM (as it depends on the entire family of confusion matrices computed at all possible thresholds $(p,t) \in [0,1]^2$), AUROC is widely used to measure performance of classifiers across the classification thresholds. Here, we show that our Theorem~\ref{thm:generalized_bayes} extends naturally from CMMs to AUROC.

We begin by noting that, for any $x \in [0, 1]$, the performance measure $\truepositive \cdot 1\{\falsepositive \leq x\}$ is a CMM. Therefore, letting $\eta^*$ denote the true regression function, by Theorem~\ref{thm:generalized_bayes}, there exists a threshold $(p,t) \in [0, 1]^2$ such that the regression-thresholding classifier $\hat Y_{p,t,\eta^*} \in \argmax_{\hat Y \in \SC} \truepositive \cdot 1\{\falsepositive \leq x\}$; i.e., $\hat Y_{p,t,\eta^*}$ maximizes $\truepositive \cdot 1\{\falsepositive \leq x\}$ over all stochastic classifiers. By definition of ROC (Eq.~\eqref{eq:ROC}), it follows that, for any $x \in [0, 1]$,
\[\eta^* \in \argmax_{\eta : \X \to [0, 1]} \operatorname{ROC}_\eta(x),\]
and, by definition AUROC (Eq.~\eqref{eq:AUROC}), it then follows that
\[\eta^* \in \argmax_{\eta : \X \to [0, 1]} \operatorname{AUROC}_\eta.\]
To conclude, we have shown that thresholding the true regression function is optimal not only under any CMM but also under AUROC. A identical argument can be made for other performance measures, such as the area under the precision-recall curve (AUPRC), that aggregate CMMs across multiple classification thresholds.

\section{Relative Performance Guarantees in terms of the Generalized Bayes Classifier}
\label{app:relative_performance_guarantees}

In this Appendix, we prove Lemmas~\ref{lemma:approximation_error} and \ref{lemma:estimation_error}, as well as their consequence, Corollary~\ref{corr:CMM_error_decomposition}. Also, in Section~\ref{app:lipschitz_constants}, we demonstrate, in a few key examples, how to compute the Lipschitz constant used in Corollary~\ref{corr:CMM_error_decomposition}.

We begin with the proof of Lemma~\ref{lemma:approximation_error}, which, at a given threshold $(p, t)$, bounds the difference between the confusion matrices of the true regression function $\eta$ and an estimate $\eta'$ of $\eta$. We restate the result for the reader's convenience:
\begin{customlemma}{\ref{lemma:approximation_error}}
    Let $p,t \in [0, 1]$ and let $\eta, \eta' : \X \to [0, 1]$. Then,
    \begin{equation}
        \left\| C_{\hat Y_{p,t,\eta}} - C_{\hat Y_{p,t,\eta'}} \right\|_\infty
        \leq \prob \left[ |\eta(X) - t| \leq \left\|\eta - \eta'\right\|_\infty \right].
        \label{lemma:approximation_error_app_result}
    \end{equation}
    \label{lemma:approximation_error_app}
\end{customlemma}

\begin{proof}
    For the true negative \edit{probability}, we have
    \begin{align*}
        \left| \truenegative_{\hat Y_{p,t,\eta}} - \truenegative_{\hat Y_{p,t,\eta'}} \right|
        & = \left| \prob \left[ Y = 0, \eta'(X) \leq t < \eta(X) \right]
          - \prob[ Y = 0, \eta(X) \leq t < \eta'(X)] \right| \\
        & \leq \prob \left[ |\eta(X) - t| \leq \|\eta - \eta'\|_\infty \right].
    \end{align*}
    This type of inequality is standard and follows from the fact that, if \(t\) lies between \(\eta\) and \(\eta'\), then the difference of \(\eta\) and \(t\) is necessarily less than \(\eta\) and \(\eta'\).
    Repeating this calculation for the true positive, false positive, and false negative probabilities gives~\eqref{lemma:approximation_error_app_result}.
\end{proof}

Note that, in the presence of degree $r$ Uniform Class Imbalance (see Section~\ref{subsec:uniform_class_imbalance}), one can obtain a tighter error bound $r\prob \left[ |\eta(X) - t| \leq \|\eta - \eta'\|_\infty \right]$ for  the true positive and false negative probabilities because, for all $x \in \X$, $\prob[Y = 1|X = x] \leq r$. However, the weaker bound~\eqref{lemma:approximation_error_app_result} simplifies the exposition.

We now turn to proving Lemma~\ref{lemma:estimation_error}, which we use to bound the maximum difference between the empirical and true confusion matrices of a regression-thresholding classifier over thresholds $(p, t)$. Specifically, we will use this result to bound the difference in confusion matrices between the optimal threshold $(p^*, t^*)$ and the threshold $(\hat p, \hat t)$ selected by maximizing the empirical CMM. We actually prove a more general version of Lemma~\ref{lemma:estimation_error}, for arbitrary classifiers, based on the following definition:

\begin{definition}[Stochastic Growth Function]
    Let $\F$ be a family of $[0, 1]$-valued functions on $\X$. The \emph{stochastic growth function $\Pi_\F : \mathbb{N} \to \mathbb{N}$}, defined by
    \[\Pi_\F(n) := \max_{\substack{x_1,...,x_n \in \X,\\z_1,...,z_n \in [0, 1]}} \left| \left\{ \left( 1\{f(x_i) > z_i\} \right)_{i = 1}^n : f \in \F \right\} \right|
    \quad \text{ for all } \quad n \in \mathbb{N},\]
    is the maximum number of distinct classifications of $n$ points $x_1,...,x_n$ by a stochastic classifier $\hat Y$ with $(x \mapsto \E[\hat Y(x)]) \in \F$ and randomness given by $z_1,...,z_n$.
    \label{def:stochastic_growth_function}
\end{definition}

Definition~\ref{def:stochastic_growth_function} generalizes the growth function~\citep{mohri2018foundations}, a classical measure of the complexity of a hypothesis class originally due to \citet{vapnik2015uniform}, to non-deterministic classifiers. Importantly for our purposes, one can easily bound the stochastic growth function of regression-thresholding classifiers:

\begin{example}[Stochastic Growth Function of Regression-Thresholding Classifiers]
    Suppose
    \[\F = \left\{f : \X \to [0, 1] \middle| \text{ for some } p, t \in [0, 1], f(x) = p \cdot 1\{\eta(x) = t\} + 1\{\eta(x) > t\} \text{ for all } x \in \X \right\},\]
    so that $\{\hat Y_{f, \eta} : f \in \F\}$ is the class of regression-thresholding classifiers. Any set of points $(x_1,z_1),...,(x_n,z_n)$, can be sorted in increasing order by $\eta(x)$'s, breaking ties in decreasing order by $z$'s. Having sorted the points in this way, $\{f(x) > z\} = 0$ for the first $j$ points and $\{f(x) > z\} = 1$ for the remaining $n - j$ points, for some $j \in [n] \cup \{0\}$. Thus, $\Pi_\F(n) = n + 1$.
    \label{example:regression_thresholding_stochastic_growth_function}
\end{example}

We will now prove the following result, from which, together with Example~\ref{example:regression_thresholding_stochastic_growth_function}, Lemma~\ref{lemma:estimation_error} follows immediately:

\begin{customlemma}{\ref{lemma:estimation_error}}[Generalized Version]
    Let $\F$ be a family of $[0, 1]$-valued functions on $\X$. Then, with probability at least $1 - \delta$,
    \[\sup_{f \in \F} \left\| \hat C_{\hat Y_f} - C_{\hat Y_f} \right\|_\infty
        \leq \sqrt{\frac{8}{n} \log \frac{32\Pi_\F(2n)}{\delta}}.\]
    \label{lemma:estimation_error_app}
\end{customlemma}


Before proving Lemma~\ref{lemma:estimation_error}, we note a standard symmetrization lemma, which allows us to replace the expectation of $\hat\truenegative_{\hat Y_{p,t,\eta}}$ with its value on an independent, identically distributed ``ghost sample''.
\begin{lemma}[Symmetrization; Lemma 2 of \citet{bousquet2003introduction}]
    Let $X$ and $X'$ be independent realizations of a random variable with respect to which $\mathcal{F}$ is a family of integrable functions. Then, for any $\epsilon > 0$,
    \[\prob \left[ \sup_{f \in \mathcal{F}} f(X) - \E f(X) > \epsilon \right]
        \leq 2\prob \left[ \sup_{f \in \mathcal{F}} f(X) - f(X') > \frac{\epsilon}{2} \right].\]
    \label{lemma:symmetrization}
\end{lemma}
We now use this lemma to prove Lemma~\ref{lemma:estimation_error}.

\begin{proof}
    To facilitate analyzing the stochastic aspect of the classifier $\hat Y_{f,\eta}$, let $Z_1,...,Z_n \stackrel{IID}{\sim} \operatorname{Uniform}([0, 1])$, such that $\hat Y_{f,\eta}(X_i) = 1\{Z_i < f(\eta((X_i))\}$.
    
    Now suppose that we have a ghost sample $(X_1',Y_1',Z_1'),...,(X_n',Y_n',Z_n')$. Let $\hat\truenegative'_{\hat Y_{f,\eta}}$ denote the empirical true negative \edit{probability} computed on this ghost sample, and let $\hat\truenegative^{(i)}_{\hat Y_{f,\eta}}$ denote the empirical true negative \edit{probability} computed on
    \[(X_1,Y_1,Z_1),...,(X_{i-1},Y_{i-1}Z_{i-1}),(X_i',Y_i',Z_i'),(X_{i+1},Y_{i+1},Z_{i+1}),...(X_n,Y_n,Z_n)\]
    (i.e., replacing only the $i^{th}$ sample with its ghost).
    By the Symmetrization Lemma,
    \begin{align}
        \prob \left[ \sup_{f \in \F} \hat\truenegative_{\hat Y_{f,\eta}} - \E \hat\truenegative_{\hat Y_{f,\eta}} > \epsilon \right]
        \notag
        & \leq 2\prob \left[ \sup_{f \in \F} \hat\truenegative_{\hat Y_{f,\eta}} - \hat\truenegative'_{\hat Y_{f,\eta}} > \epsilon/2 \right] \\
        \notag
        & \leq 2\Pi_\F(2n) \sup_{f \in \F} \prob \left[ \hat\truenegative_{\hat Y_{f,\eta}} - \hat\truenegative'_{\hat Y_{f,\eta}} > \epsilon/2 \right] \\
        \label{ineq:symmetrization}
        & \leq 4\Pi_\F(2n) \sup_{f \in \F} \prob \left[ \hat\truenegative_{\hat Y_{f,\eta}} - \E \hat\truenegative_{\hat Y_{f,\eta}} > \epsilon/4 \right],
    \end{align}
    where the second inequality is a union bound over the $\Pi_\F(2n)$ distinct classifications of $2n$ points that can be assigned by $\hat Y_{f,\eta}$ with $f \in \F$, and the last inequality is from the fact that $\hat\truenegative_{\hat Y_{f,\eta}}$ and $\hat\truenegative'_{\hat Y_{f,\eta}}$ are identically distributed and the algebraic fact that, if $a - b > \epsilon$, then either $a - c > \epsilon/2$ or $b - c > \epsilon/2$.
    
    For any particular $f \in \F$, by McDiarmid's inequality~\citep{mcdiarmid1998concentration},
    \begin{equation}
        \prob \left[ \hat\truenegative_{\hat Y_{f,\eta}} - \E\hat\truenegative_{\hat Y_{f,\eta}} > \epsilon/4 \right]
        \leq e^{-n\epsilon^2/8},
        \label{ineq:application_of_mcdiarmid}
    \end{equation}
    since, for any $i \in [n]$,
    \[\left| \hat\truenegative_{\hat Y_{f,\eta}} - \hat\truenegative^{(i)}_{\hat Y_{f,\eta}} \right|
        = \frac{1}{n} \left| 1\left\{Y_i = \hat Y_{f,\eta}(X_i) = 0 \right\}
        - 1\left\{Y_i' = \hat Y_{f,\eta}(X_i') = 0 \right\} \right|
        \leq \frac{1}{n}.\]
    Plugging Inequality~\eqref{ineq:application_of_mcdiarmid} into Inequality~\eqref{ineq:symmetrization} gives
    \[\prob \left[ \sup_{f \in \F} \hat\truenegative_{\hat Y_{f,\eta}} - \E \hat\truenegative_{\hat Y_{f,\eta}} > \epsilon \right]
        \leq 4 \Pi_\F(2n) e^{-n\epsilon^2/8}.\]
    Repeating this argument with $-\hat\truenegative$ instead of $\hat\truenegative$, as well as with $\hat\truepositive$, $\hat\falsenegative$, $\hat\falsepositive$ and their negatives, and taking a union bound over these $8$ cases, gives the desired result.
\end{proof}

Finally, we will use these two lemmas, together with the margin and Lipschitz assumptions, to prove Corollary~\ref{corr:CMM_error_decomposition}, which bounds the sub-optimality of the trained classifier, relative to the generalized Bayes classifier, in terms of the desired CMM.

\begin{customcorollary}{\ref{corr:CMM_error_decomposition}}
    Let $\eta : \X \to [0, 1]$ denote the true regression function, and let $\hat\eta : \X \to [0, 1]$ denote any empirical regressor.
    Let
    \[\left( \hat p, \hat t \right) := \argmax_{(p, t) \in [0, 1]^2} M \left( \hat C_{\hat Y_{p,t,\hat\eta}} \right)
    \quad \text{ and } \quad
    \left( p^*, t^* \right) := \argmax_{(p, t) \in [0, 1]^2} M \left(C_{\hat Y_{p,t,\eta}} \right)\]
    denote the empirically selected and true optimal thresholds, respectively. Suppose that $M$ is Lipschitz continuous with constant $L_M$ with respect to the uniform ($\L_\infty$) metric on $\C$. Finally, suppose that $P_X$ and $\eta$ satisfies a $(C, \beta)$-margin condition around $t^*$. Then, with probability at least $1 - \delta$,
    \begin{align}
        \notag
        M\left(C_{\hat Y_{p,t,\eta}}\left(p^*, t^*\right)\right) - M\left(C_{\hat Y_{p,t,\hat\eta}}\left(\hat p, \hat t\right)\right)
        & \leq L_M \left( C\left\|\eta - \hat\eta\right\|_\infty^\beta + 2 \sqrt{\frac{8}{n} \log \frac{32(2n + 1)}{\delta}} \right).
    \end{align}
    \label{corr:CMM_error_decomposition_app}
\end{customcorollary}

\begin{proof}
    First, note that
    \begin{align*}
        M\left(C_{\hat Y_{p^*, t^*,\eta}}\right) - M\left(C_{\hat Y_{\hat p, \hat t,\hat\eta}}\right)
        & \leq M\left(C_{\hat Y_{p^*, t^*,\eta}}\right) - M\left(C_{\hat Y_{p^*, t^*,\hat\eta}}\right) \\
        & + M\left(C_{\hat Y_{p^*, t^*,\hat\eta}}\right) - M\left(\hat C_{\hat Y_{p^*, t^*,\hat\eta}}\right) \\
        & + M\left(\hat C_{\hat Y_{\hat p, \hat t,\hat\eta}}\right) - M\left(C_{\hat Y_{\hat p, \hat t,\hat\eta}}\right),
    \end{align*}
    since, by definition of $(\hat p, \hat t)$,
    \[M\left(\hat C_{\hat Y_{p^*,t^*,\hat\eta}}\right) - M\left(\hat C_{\hat Y_{\hat p,\hat t,\hat\eta}} \right) \leq 0;\]
    this term sits between the second and third lines above.
    By the Lipschitz assumption,
    \begin{align}
        \notag
        & M\left(C_{\hat Y_{p^*, t^*,\eta}}\right) - M\left(C_{\hat Y_{\hat p, \hat t,\hat\eta}}\right) 
        \\
        \label{term:approximation_error}
         &\leq L_M \bigg( \left\| C_{\hat Y_{p^*, t^*,\eta}} - C_{\hat Y_{p^*, t^*,\hat\eta}} \right\|_\infty
        \\ \label{term:estimation_error_term1} 
        & \hspace{2.62em}
         + \left\| C_{\hat Y_{p^*, t^*,\hat\eta}} - \hat C_{\hat Y_{p^*, t^*,\hat\eta}} \right\|_\infty 
        \\
        \label{term:estimation_error_term2} 
         & \hspace{2.62em}
         + \left\| \hat C_{\hat Y_{\hat p, \hat t,\hat\eta}} - C_{\hat Y_{\hat p, \hat t,\hat\eta}} \right\|_\infty \bigg).
    \end{align}
    Corollary~\ref{corr:CMM_error_decomposition} follows by applying Lemma~\ref{lemma:approximation_error} and the $(C,\beta)$-margin condition to \eqref{term:approximation_error} and applying Lemma~\ref{lemma:estimation_error} to both terms \eqref{term:estimation_error_term1} and \eqref{term:estimation_error_term2}.
\end{proof}

\subsection{Lipschitz constants for some common CMMs}
\label{app:lipschitz_constants}
Corollary~\ref{corr:CMM_error_decomposition} assumed that the CMM $M$ was Lipschitz continuous with respect to the $\sup$-norm on confusion matrices.
In this section, we show how to compute appropriate Lipschitz constants for several simple example CMMs. We begin with a simple example:

\begin{example}[Weighted Accuracy]
For a fixed $w \in (0, 1)$, the $w$-weighted accuracy is given by $M(C) = (1 - w) \truepositive + w \truenegative$. In this case, $M$ clearly has Lipschitz constant $L_M = \max\{w, 1 - w\}$.
\end{example}

For the remainder of this section (only), we will use $P := \E[Y]$ to denote the positive \edit{probability} of the true labels and $\hat P := \frac{1}{n} \sum_{i = 1}^n Y_i$ to denote the empirical positive \edit{probability} of the true labels.
Many CMMs of interest, such as Recall and $F_\beta$ scores, are not Lipschitz continuous over all of $\C$. Fortunately, inspecting the proof of Corollary~\ref{corr:CMM_error_decomposition}, it suffices for the CMM $M$ to be Lipschitz continuous on the line segments between three specific pairs of confusion matrices, given in Eqs.~\eqref{term:approximation_error}, \eqref{term:estimation_error_term1}, and \eqref{term:estimation_error_term2}. Deriving the appropriate Lipschitz constants is a bit more complex, and we demonstrate here how to derive them for the specific CMMs of Recall and $F_\beta$ scores.

Of the six confusion matrices in Eqs.~\eqref{term:approximation_error}, \eqref{term:estimation_error_term1}, and \eqref{term:estimation_error_term2}, four are true confusion matrices, while the other two are empirical. The four true confusion matrices have the same positive \edit{probability} $\truepositive + \falsenegative = P$, which is a function of the true distribution of labels. The two empirical confusion matrices have the positive \edit{probability} $\hat\truepositive + \hat\falsenegative = \hat P$, which is a function of the data. By a multiplicative Chernoff bound, with probability at least $1 - e^{-nP/8}$, $\hat P \geq P/2$. Thus, with high probability, it suffices for the CMM $M$ to be Lipschitz continuous over confusion matrices with positive \edit{probability} at least $P/2$. For Recall and $F_\beta$ scores, this gives the following Lipschitz constants:

\begin{example}[Recall]
    Recall is given by $M(C) = \frac{\truepositive}{\truepositive + \falsenegative} = \frac{\truepositive}{P}$. Thus, $M$ is Lipschitz continuous with constant $L_M = \frac{2}{P}$ over the confusion matrices in Eqs.~\eqref{term:approximation_error}, \eqref{term:estimation_error_term1}, and \eqref{term:estimation_error_term2}.
    \label{ex:recall_lipschitz}
\end{example}

\begin{example}[$F_\beta$ Score]
    For $\beta \in (0, \infty)$, the $F_\beta$ score is given by
    \[M(C)
      = \frac{(1 + \beta^2) \truepositive}{(1 + \beta^2) \truepositive + \falsepositive + \beta^2 \falsenegative}
      = \frac{(1 + \beta^2) \truepositive}{\truepositive + \falsepositive + \beta^2 P}.\]
    Hence,
    \[\left| \frac{\partial}{\partial\truepositive} M(C) \right|
      = (1 + \beta^2) \frac{\falsepositive + \beta^2 P}{\left( \truepositive + \falsepositive + \beta^2 P \right)^2}
      \leq \frac{1 + \beta^2}{\beta^2 P},\]
    while, since $\truepositive \leq P$,
    \[\left| \frac{\partial}{\partial\falsepositive} M(C) \right|
      = (1 + \beta^2) \frac{\truepositive}{\left( \truepositive + \falsepositive + \beta^2 P \right)^2}
      \leq \frac{1 + \beta^2}{\beta^4 P}.\]
    Hence, $M$ is Lipschitz continuous with constant $\frac{2(1 + \beta^2)}{P} \max \left\{ \beta^{-2}, \beta^{-4} \right\}$ over the confusion matrices in Eqs.~\eqref{term:approximation_error}, \eqref{term:estimation_error_term1}, and \eqref{term:estimation_error_term2}.
    \label{ex:f_beta_lipschitz}
\end{example}

As Examples~\ref{ex:recall_lipschitz} and \ref{ex:f_beta_lipschitz} demonstrate, the Lipschitz constants of some CMMs can become large when the proportion $P$ is positive samples is small. In particular, when $P \in O \left( \sqrt{\frac{\log n}{n}} \right)$, the $\asymp L_M \sqrt{\frac{\log(n/\delta)}{n}}$ term of Corollary~\ref{corr:CMM_error_decomposition} fails to vanish as $n \to \infty$. We believe that some loss of convergence rate is inevitable if $P \to 0$ as $n \to \infty$, due to the inherent instability of such metrics, but further work is needed to understand if the rates given by Corollary~\ref{corr:CMM_error_decomposition} are optimal under these metrics.
See also \citet{dembczynski2017consistency} for detailed discussion of Lipschitz constants of many common CMMs.

\section{Bounds on Uniform Error of the Nearest Neighbor Regressor}
\label{app:UniformConvergenceProof}

In this appendix, we prove our upper bound on the uniform risk of the $k$NN regressor (Theorem~\ref{thm:unif_convergence}), as well as the corresponding minimax lower bound (Theorem~\ref{theorem:UniformErrorLowerBound}).

\subsection{Upper Bounds}
\label{app:knn_upper_bound_proofs}
Here, we prove Theorem~\ref{thm:unif_convergence}, our upper bound on the uniform error of the $k$-NN regressor, restated below:
\begin{customthm}{\ref{thm:unif_convergence}}
    Under Assumptions~\ref{assumption:denseCovariates} and \ref{assumption:holderContinuity}, whenever $k / n \leq p_*(\epsilon^*)^d / 2$,
    for any
    $\delta > 0$, with probability at least $1 - N\left( \left( 2k / (p_* n) \right)^{1/d} \right) e^{-k/4} - \delta$,
    we have the uniform error bound
    \begin{equation}
        \left\|\eta - \hat\eta\right\|_\infty \leq 2^\alpha Lr\left( \frac{2k}{p_* n} \right)^{\alpha/d}
        + \frac{2}{3k} \log \frac{2 S(n)}{\delta} + \sqrt{\frac{2r}{k} \log \frac{2 S(n)}{\delta}}.
    \label{eq:uniform_error_bound_app}
    \end{equation}
\end{customthm}

\begin{proof}
For any $x \in \X$, let
\[\tilde \eta_k(x) := \frac{1}{k} \sum_{j = 1}^k \eta(X_{\sigma_j(x)})\] denote the mean of the true regression function over the $k$ nearest neighbors of $x$. By the triangle inequality,
\[\|\eta - \hat\eta\|_\infty
  \leq \|\eta - \tilde \eta_k\|_\infty
     + \|\tilde \eta_k - \hat \eta\|_\infty,\]
wherein $\|\eta - \tilde \eta_k\|_\infty$ captures bias due to smoothing and $\|\tilde \eta_k - \hat \eta\|_\infty$ captures variance due to label noise. We separately show that, with probability at least $1 - N \left( \left( \frac{2k}{p_* n} \right)^{1/d} \right) e^{-k/4}$,
\[\left\| \eta - \tilde \eta_k \right\|_\infty
  \leq 2^\alpha Lr \left( \frac{2k}{p_* n} \right)^{\alpha/d},\]
and that, with probability at least $1 - \delta$,
\[\|\tilde \eta_k - \hat \eta\|_\infty
  \leq \frac{2}{3k} \log \frac{2 S(n)}{\delta}
  + \sqrt{\frac{2r}{k} \log \frac{2 S(n)}{\delta}}.\]

\paragraph{Bounding the smoothing bias}

Fix some $r > 0$ to be determined, and let $\{B_r(z_1),...,B_r(z_{N(r)})\}$ be a covering of $(\X, \rho)$ by $N(r)$ balls of radius $r$, with centers $z_1,...,z_{N(r)} \in \X$.

By the lower bound assumption on $P_X$, each $P_X(B_r(z_j)) \geq p_* r^d$. Therefore, by a multiplicative Chernoff bound, with probability at least $1 - N(r) e^{-p_* n r^d/8}$, each $B_r(z_j)$ contains at least $p_* n r^d/2$ samples. In particular, if $r \geq \left( \frac{2k}{p_* n} \right)^{1/d}$, then each $B_k$ contains at least $k$ samples, and it follows that, for every $x \in \X$, $\rho(x, X_{\sigma_k(x)}) \leq 2r$.
Thus, by H\"older continuity of $\eta$,
\[
\left| \eta(x) - \tilde \eta_k(x) \right|
  = \left| \eta(x) - \frac{1}{k} \sum_{j = 1}^k \eta(X_{\sigma_j(x)}) \right|
  \leq \frac{1}{k} \sum_{j = 1}^k \left| \eta(x) - \eta(X_{\sigma_j(x)}) \right|
  \leq L(2r)^\alpha.\]
Finally, if $\frac{k}{n} \leq \frac{p_*}{2} (r^*)^d$, then we can let $r = \left( \frac{2k}{p_* n} \right)^{1/d}$.

\paragraph{Bounding variance due to label noise}

Let $\Sigma := \{\sigma(x) \in [n]^k : x \in \X \}$ denote the set of possible $k$-nearest neighbor index sets. One can check from the definition of the shattering coefficient that $|\Sigma| \leq S(n)$.

For any $\sigma \in [n]^k$, let $Z_\sigma := \sum_{j = 1}^k Y_{\sigma_j}$ and let $\mu_\sigma := \E\left[ Z_\sigma \right]$. Note that the conditional random variables $Y_{\sigma_j}|X_1,...,X_n$ have conditionally independent Bernoulli distributions with means $\E[Y_{\sigma_j}|X_1,...,X_n] = \eta(X_{\sigma_j})$ and variances $\E \left[ \left( Y_{\sigma_j} - \eta(X_{\sigma_j}) \right)^2 |X_1,...,X_n \right] = \eta(X_{\sigma_j}) (1 - \eta(X_{\sigma_j})) \leq r$.
Therefore, by Bernstein's inequality (Eq. (2.10) of \cite{boucheron2013}), for any $\epsilon > 0$,
\begin{equation}
    \prob \left[ |Z_\sigma/k - \mu_\sigma| \geq \epsilon \right] \leq 2 \exp \left( -\frac{k\epsilon^2}{2(r + \epsilon/3)} \right).
    \label{ineq:bernstein}
\end{equation}
Moreover, for any $x \in \X$, $\mu_{\sigma(x)} = \tilde \eta_k(x)$ and $Z_{\sigma(x)}/k = \hat \eta(x)$. Hence, by a union bound over $\sigma$ in $\Sigma$,
\begin{align*}
    \prob \left( \sup_{x \in \X} \left| \tilde \eta_k(x) - \hat \eta(x) \right| > \epsilon \middle| X_1,...,X_n \right)
    & = \prob \left( \sup_{x \in \X} \left| \mu_{\sigma(x)} - Z_{\sigma(x)}/k\right| > \epsilon \middle| X_1,...,X_n \right) \\
    & \leq \prob \left( \sup_{\sigma \in \Sigma} \left| \mu_\sigma - Z_\sigma/k\right| > \epsilon \middle| X_1,...,X_n \right) \\
    & \leq |\Sigma| \sup_{\sigma \in \Sigma} \prob \left( \left| \mu_\sigma - Z_\sigma/k\right| > \epsilon \middle| X_1,...,X_n \right) \\
    & \leq 2S(n) \exp \left( -\frac{k\epsilon^2}{2(r + \epsilon/3)} \right).
\end{align*}
Since the right-hand side is independent of $X_1,...,X_n$, the unconditional bound
\[\prob \left( \sup_{x \in \X} \left\| \tilde \eta_k(x) - \hat \eta(x) \right\|_\infty > \epsilon \right)
  \leq 2S(n) \exp \left( -\frac{k\epsilon^2}{2(r + \epsilon/3)} \right)\]
follows. Plugging in
\[\epsilon
  = \frac{1}{3k} \log \frac{2 S(n)}{\delta} + \sqrt{\left( \frac{1}{3k} \log \frac{2 S(n)}{\delta} \right)^2 + \frac{2r}{k} \log \frac{2 S(n)}{\delta}}
  \leq \frac{2}{3k} \log \frac{2 S(n)}{\delta} + \sqrt{\frac{2r}{k} \log \frac{2 S(n)}{\delta}}\]
and simplifying gives the final result.
\end{proof}

Recall that there is a small (polylogarithmic in $r$) gap between our upper and lower bounds. We believe that the upper bound may be slightly loose, and that this might be tightened by using a stronger concentration inequality, such as Bennett's inequality~\citep{bennett1962probability}, instead of Bernstein's inequality in Inequality~\eqref{ineq:bernstein}.


Naively applying Theorem~\ref{thm:unif_convergence} results in very slow convergence rates in high dimensions. For this reason, we close this section with a corollary of Theorem~\ref{thm:unif_convergence}, illustrating that the convergence rates provided by Theorem~\ref{thm:unif_convergence} improve if the covariates are assumed to lie on an (unknown) lower dimensional manifold:

\begin{corollary}[Implicit Manifold Case]
Suppose $Z$ is a $[0,1]^d$-valued random variable
with a density lower bounded away from $0$, and suppose that, for some Lipschitz map $T : [0,1]^d \to \R^D$, $X = T(Z)$.
Then, $N(\epsilon) \leq (2/\epsilon)^d$, and $S(n) \leq 2n^{D + 1} + 2$, and so, by Theorem~\ref{thm:unif_convergence}, $k \asymp n^{\frac{2\alpha}{2\alpha+d}} (\log n)^{\frac{d}{2\alpha+d}} r^{-\frac{d}{2\alpha + d}}$,
\[\left\|\eta - \hat\eta\right\|_\infty
  \in O_P \left( \left( \frac{\log n}{n} \right)^{\frac{\alpha}{2\alpha+d}} r^\frac{\alpha + d}{2\alpha + d} \right).\]
\end{corollary}
This shows that, if the $D$ covariates lie implicitly on a $d$-dimensional manifold, convergence rates depend on $d$, which may be much smaller than $D$.

\subsection{Lower Bounds}
\label{app:knn_lower_bound_proofs}

In this section, we prove Theorem~\ref{theorem:UniformErrorLowerBound}, our lower bound on the minimax uniform error of estimating a H\"older continuous regression function. We use a standard approach based on the following version of Fano's lemma:
\begin{lemma} (Fano's Lemma; Simplified Form of Theorem 2.5 of \citealt{tsybakov2009introduction})
    Fix a family $\P$ of distributions over a sample space $\X$ and fix a pseudo-metric $\rho : \P \times \P \to [0,\infty]$ over $\P$. Suppose there exist $P_0 \in \P$ and a set $T \subseteq \P$ such that
    \[\sup_{P \in T} D_{KL}(P,P_0)
      \leq \frac{\log |T|}{16},\]
    where $D_{KL} : \P \times \P \to [0,\infty]$ denotes Kullback-Leibler divergence.
    Then,
    \[\inf_{\hat P} \sup_{P \in \P} 
    \prob \left( \rho(P,\hat P)
      \geq \frac{1}{2} \inf_{P \in T} \rho(P,P_0) \right) \geq 1/8,\]
    where the first $\inf$ is taken over all estimators $\hat P$.
    \label{thm:tsybakov_fano}
\end{lemma}

\begin{proof}
We now proceed to construct an appropriate $P_0 \in \P$ and $T \subseteq \P$. Let $g : [-1,1]^d \to [0,1]$ defined by
\[g(x) = \left\{ \begin{array}{cc}
    \exp \left( 1 - \frac{1}{1 - \|x\|_2^2} \right) & \text{ if } \|x\|_2 < 1 \\
    0 & \text{ else }
\end{array} \right.\]
denote the standard bump function supported on $[-1,1]^d$, scaled to have $\|g\|_{\X,\infty} = 1$. Since $g$ is infinitely differentiable and compactly supported, it has a finite $\alpha$-H\"older semi-norm:
\begin{equation}
\|g\|_{\Sigma^\alpha}
  := \sup_{\ell \in \mathbb{N}^d : \|\ell\|_1 \leq \alpha} \quad \sup_{x \neq y \in \X} \quad \frac{|g^\ell(x) - g^\ell(y)|}{\|x - y\|^{\alpha - \|\ell\|_1}}
  < \infty,
\label{eqn:HolderSeminorm}
\end{equation}
where $\ell$ is any $\lfloor \beta \rfloor$-order multi-index and $g^\ell$ is the corresponding mixed derivative of $g$.
Define $M := \left( \frac{64(2\alpha + d)nr}{d \log(nr)} \right)^{\frac{1}{2\alpha + d}} \geq 1$, since $r \geq 1/n$. For each $m \in [M]^d$, define $g_m : \X \to [0,1]$ by
\[g_m (x) := g\left( Mx - \frac{2m - 1_d}{2} \right),\]
so that $\{g_m : m \in [M]^d\}$ is a grid of $M^d$ bump functions with disjoint supports.
Let $\zeta_0 \equiv \frac{1}{4}$ denote the constant-$\frac{1}{4}$ function on $\X$. Finally, for each $m \in [M]^d$, define $\zeta_m : \X \to [0,1]$ by
\begin{equation}
    \zeta_m := \zeta_0 + \min \left\{ \frac{1}{2}, \frac{L}{\|g\|_{\Sigma^\alpha}} \right\} M^{-\alpha} g_m.
    \label{eq:eta_m}
\end{equation}
Note that, for any $m \in [M]^d$,
\[\|\zeta_m\|_{\Sigma^\alpha}
  \leq L M^{-\alpha} \frac{\|g_m\|_{\Sigma^\alpha}}{\|g\|_{\Sigma^\alpha}}
  = L,\]
so that $\zeta_m$ satisfies the H\"older smoothness condition.
For any particular $\eta$, let $P_\eta$ denote the joint distribution of $(X, Y)$. Note that $P_\zeta(x, 1) = \zeta(x) \geq 1/4$. Moreover, one can check that, for all $x \geq -2/3$, $-\log(1 + x) \leq x^2 - x$. Hence, for any $x \in \X$,
\begin{align*}
    P_{\eta_m}(x, 1) \log \frac{P_{\eta_m}(x, 1)}{P_{\eta}(x, 1)}
    & = rP_{\zeta_m}(x, 1) \log \frac{P_{\zeta_m}(x, 1)}{P_{\zeta}(x, 1)} \\
    & = r\zeta_m(x) \log \frac{\zeta_m(x)}{\zeta(x)} \\
    & = -r\zeta_m(x) \log \left( 1 + \frac{\zeta(x) - \zeta_m(x)}{\zeta_m(x)} \right) \\
    & \leq r\zeta_m(x) \left( \left( \frac{\zeta(x) - \zeta_m(x)}{\zeta_m(x)} \right)^2 - \frac{\zeta(x) - \zeta_m(x)}{\zeta_m(x)} \right) \\
    & = r \left( \frac{\left( \zeta(x) - \zeta_m(x)\right)^2 }{\zeta_m(x)} - \zeta(x) + \zeta_m(x) \right) \\
    & \leq r \left( 4\left( \zeta(x) - \zeta_m(x)\right)^2 - \zeta(x) + \zeta_m(x) \right),
\end{align*}
and, similarly, since $P_\zeta(x,0) = 1 - \zeta(x) \geq 1/4$,
\[P_{r\eta_m}(x, 0) \log \frac{P_{r\eta_m}(x, 0)}{P_{r\eta}(x, 0)}
    \leq r\left( 4\left( \zeta(x) - \zeta_m(x)\right)^2 + \zeta(x) - \zeta_m(x) \right).\]
Adding these two terms gives
\begin{align*}
D_{\text{KL}}\left( P_{r\eta}^n, P_{r\eta_m}^n \right)
 & = n \left( \int_\X P_{r\eta_m}(x, 0) \log \frac{P_{r\eta}(x, 0)}{P_{r\eta_m}(x, 0)} \, dx + \int_\X P_{r\eta_m}(x, 1) \log \frac{P_{r\eta}(x, 1)}{P_{r\eta_m}(x, 1)} \, dx \right) \\
 & \leq 8nr\int_\X \left( \zeta(x) - \zeta_m(x)\right)^2 \\
 & = 8nr\|\zeta - \zeta_m\|_2^2 \\
 & \leq 2nr M^{-2\alpha} \|g_m\|_2^2 \\
 & = 2nr M^{-(2\alpha + d)} \|g\|_2^2 \\
 & = 2nr \left( \left( \frac{64 (2\alpha + d) nr}{d \log(nr)} \right)^{\frac{1}{2\alpha + d}}\right)^{-(2\alpha + d)} \|g\|_2^2 \\
 & = \frac{1}{32} \frac{d}{2\alpha + d} \|g\|_2^2 \log(nr) \\
 & \leq \frac{1}{16} \frac{d}{2\alpha + d} \left( \log(nr) - \log \log(nr) + \log \frac{64 (2\alpha + d)}{d} \right)
   = \frac{\log |[M]^d|}{16},
\end{align*}
where the second inequality comes from the definition of $\zeta_m$ (Eq.~\ref{eq:eta_m}) and the third inequality comes from the facts that $\|g\|_2^2 \leq 1$ and $\log \log x \leq \frac{1}{2} \log x$ for all $x > 1$.
Fano's lemma therefore implies the lower bound
\[\inf_{\hat \eta} \sup_{r \in (0, 1], \zeta \in \Sigma^\alpha(L)} 
\prob_{\{(X_i,Y_i)\}_{i = 1}^n \sim P_\eta^n} \left( \left\|r\zeta - r\hat \zeta \right\|_\infty \geq C \left( \frac{\log(nr)}{n} \right)^{\frac{\alpha}{2\alpha + d}} r^\frac{\alpha + d}{2\alpha + d} \right)
  \geq \frac{1}{8},\]
  where
  \[C = \frac{1}{2} \min \left\{ \frac{1}{2}, \frac{L}{\|g\|_{\Sigma^\alpha}} \right\} \left( \frac{d}{64(2\alpha + d)} \right)^{\frac{\alpha}{2\alpha + d}}.\]
\end{proof}

\section{Efficient Computation of the Optimal Stochastic Threshold}
\label{app:computation}

Although the focus of this paper is on \emph{statistical} properties of regression-thresholding classifiers, we note that, given an estimate $\hat \eta$ of the regression function, the empirically optimal stochastic threshold $(\hat p, \hat t)$, i.e., that which maximizes $M(\hat C_{\hat Y_{\hat \eta,\hat p, \hat t}})$, can be efficiently computed. In this appendix, we describe a simple algorithm for doing so. The key insight is that, because $(\hat p, \hat t)$ is used to threshold the observed empirical class probabilities $\hat \eta(X_1),...,\hat \eta(X_n)$ before computing $M$, $M(\hat C_{\hat Y_{\hat \eta,p, t}})$ only needs to be computed at the $n$ values of $\hat \eta$ actually observed in the data.

We also note that, while, by Corollary~\ref{corr:CMM_error_decomposition}, one can safely use the original training dataset to compute $(\hat p, \hat t)$, one can also safely use a much smaller subset of the data, since the rate of convergence in Lemma~\ref{lemma:estimation_error} is quite fast in $n$.

\begin{algorithm2e}[htb]
    \DontPrintSemicolon
        \KwInput{Estimated regression function $\hat\eta$,
                 training covariate samples $X_1,...,X_n$,
                 CMM $M$.}
        \KwOutput{Estimated optimal stochastic threshold $(\hat p, \hat t)$}
        Sample $Z_1,...,Z_n \stackrel{IID}{\sim} \operatorname{Uniform}([0, 1])$ \\
        $e_1,...,e_n \leftarrow \hat\eta(X_1),...,\hat\eta(X_n)$ \\
        $(e_1, Z_1),...,(e_n, Z_n) \leftarrow \mathsf{LexicographicSort}((e_1,Z_n),...,(e_n,Z_n))$ \\
        $\truepositive \leftarrow \frac{1}{n} \sum_{i = 1}^n Y_i$ \\
        $\falsepositive \leftarrow 1 - \text{TP}$ \\
        $\truenegative, \falsenegative \leftarrow 0$ \\
        $(\hat p, \hat t) \leftarrow (0, 0)$ \\
        $M_{\text{max}} \leftarrow M \left( \begin{bmatrix}
                \truenegative & \falsepositive \\
                \falsenegative & \truepositive
            \end{bmatrix} \right)$ \\
        \For{$i = 1; i <= n; i++$}{
            $\truepositive \leftarrow \truepositive - Y_i/n$ \\
            $\falsepositive \leftarrow \truepositive - (1 - Y_i)/n$ \\
            $\truenegative \leftarrow \truenegative + (1 - Y_i)/n$ \\
            $\falsenegative \leftarrow \falsenegative + Y_i/n$ \\
            $M_{\text{new}} \leftarrow M \left( \begin{bmatrix}
                \truenegative & \falsepositive \\
                \falsenegative & \truepositive
            \end{bmatrix} \right)$ \\
            \uIf{$M_{\text{new}} > M_{\text{max}}$}{
                 $(\hat p, \hat t) \leftarrow (e_i, Z_i)$ \\
                 $M_{\text{max}} \leftarrow M_{\text{new}}$ \\
            }
        }
        \Return{$(\hat p, \hat t)$}
    \caption{Efficient threshold-optimization algorithm.}
    \label{alg:threshold_selection}
\end{algorithm2e}

For large $n$, the runtime of Algorithm~\ref{alg:threshold_selection} is dominated by Line 3, which involves lexicographically sorting $n$ pairs. This can be done in $O(n \log n )$ time using standard comparison-based sorting algorithms. Hence, the overall runtime of Algorithm~\ref{alg:threshold_selection} is $O(n \log n)$.

\section{Further Experimental Details}
\label{app:further_experimental_details}
Experiments were run using the \texttt{numpy} and \texttt{scikit-learn} packages in Python 3.9, on a machine running Ubuntu 20.04 with an Intel Core i5-9600 CPU and 64 gigabytes of memory. Each experiment took about 10 minutes to run. Python code and instructions for reproducing Figures~\ref{fig:experiment1} and \ref{fig:experiment2} are available at \url{https://gitlab.tuebingen.mpg.de/shashank/imbalanced-binary-classification-experiments}.

\section{Experiments with Real Data: Case Study in Credit Card Fraud Detection}
\label{app:credit_card}

In this section, we explore theoretical predictions from the main paper in a real dataset, the Kaggle Credit Card Fraud Detection dataset (available at \url{https://www.kaggle.com/datasets/mlg-ulb/creditcardfraud} under an Open Database License (ODbL)), a widely used benchmark dataset for imbalanced classification. This dataset contains $29$ continuous features (computed via PCA from an underlying set of features) for each of 284,807 credit card transactions, of which $492$ ($0.172\%$) are labeled as fraudulent, and the remaining are assumed to be non-fraudulent. The supervised learning task is to predict whether a credit card transaction is fraudulent, given its $29$ PCA features. Due to computational limitations, we down-sampled the negative set (non-fraudulent transactions) by a factor of $0$ before conducting our experiments; however, we expect our main observations to hold on the full dataset as well.
We also $Z$-scored each feature (to have mean $0$ and variance $1$).

\begin{figure}
    \centering
    \includegraphics[width=0.8\linewidth]{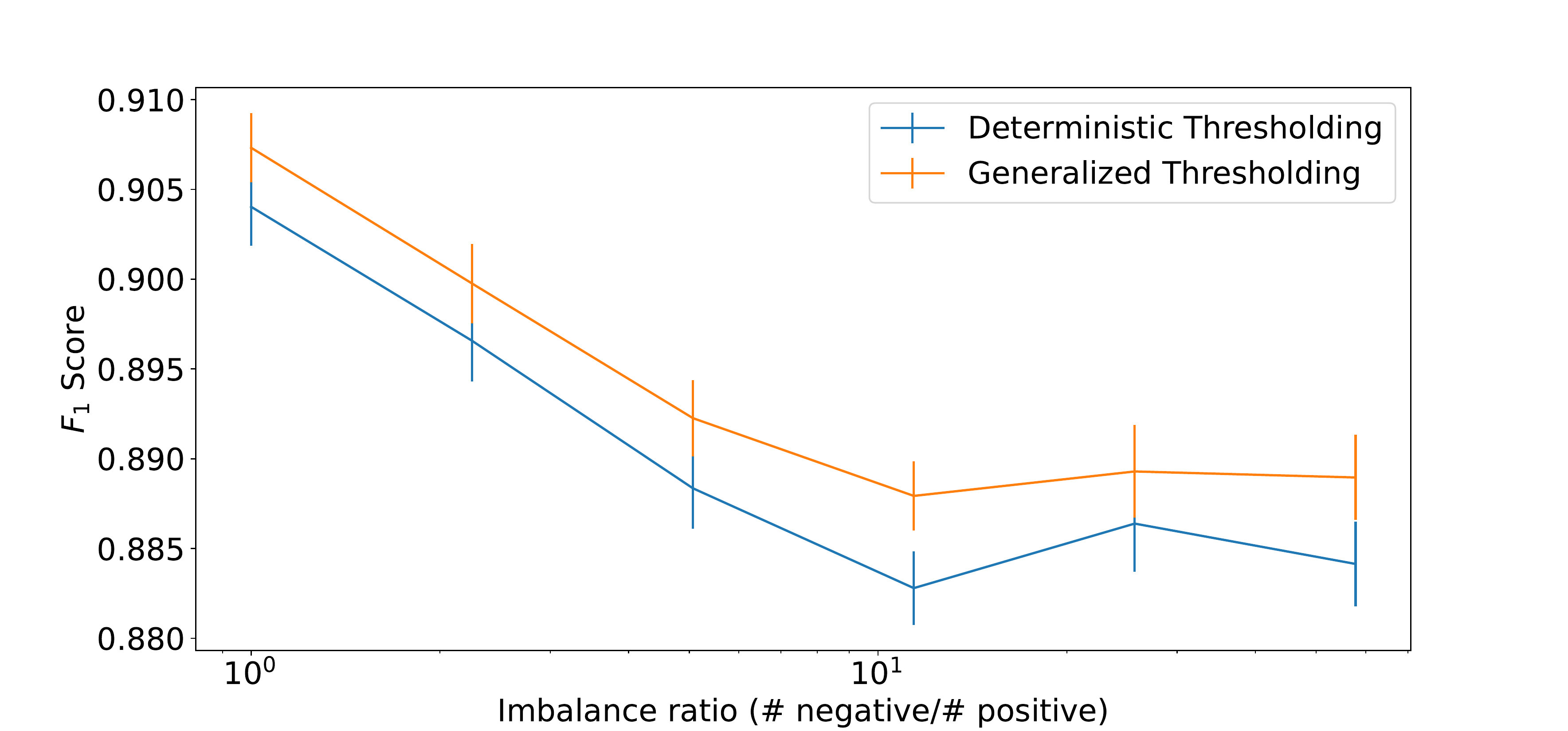}
    \caption{Mean $F_1$ scores (over $100$ random training/validation/test splits) of optimal deterministic and stochastic thresholding nearest neighbor classifiers, on the credit card fraud dataset, at various degrees of class imbalance. Error bars denote standard errors, computed over the $100$ random training/validation/test splits.}
    \label{fig:credit_card_f1_over_class_imbalance}
\end{figure}

\begin{figure}
    \centering
    \includegraphics[width=0.8\linewidth]{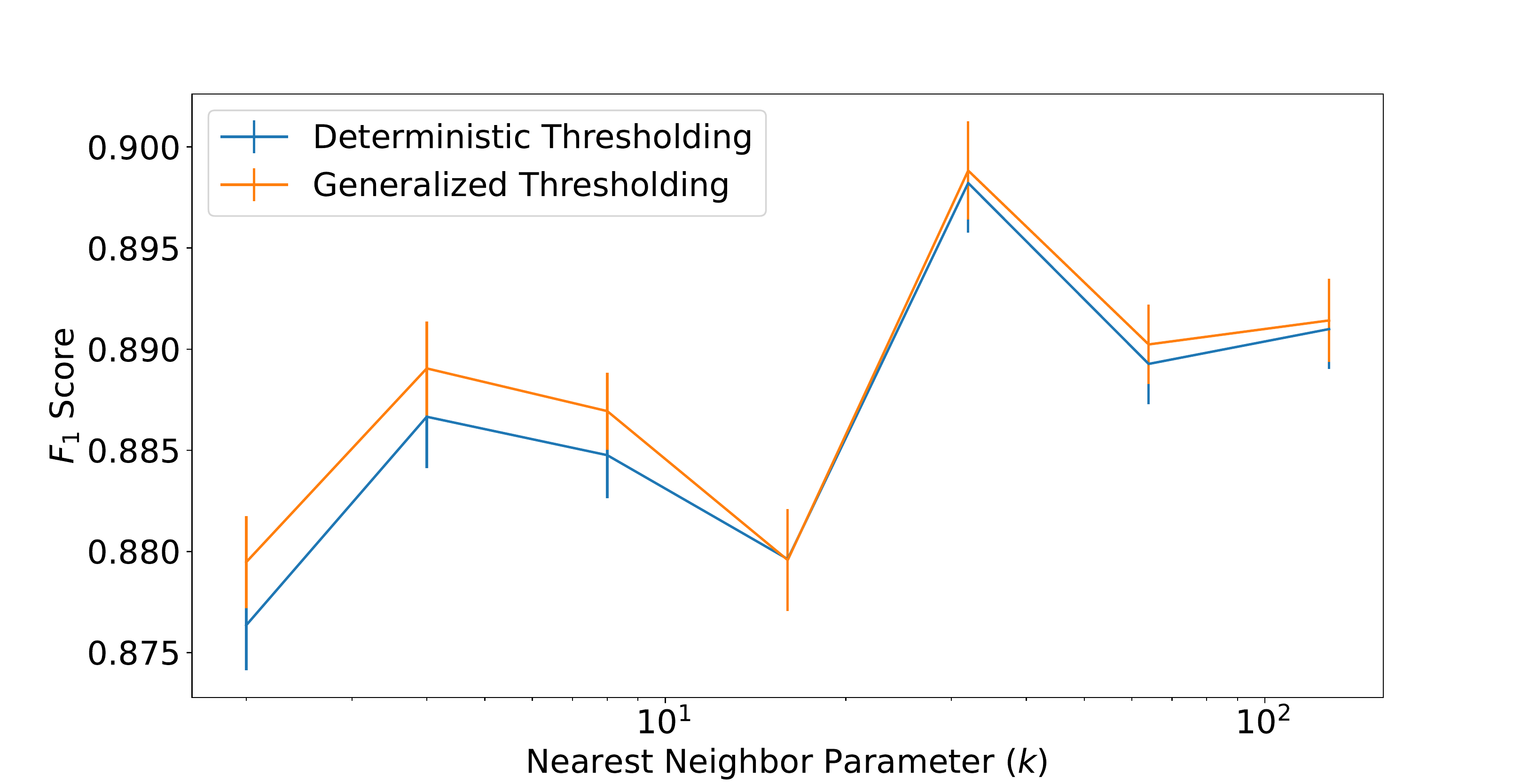}
    \caption{Mean $F_1$ scores (over $100$ random training/validation/test splits) of optimal deterministic and stochastic thresholding nearest neighbor classifiers, on the credit card fraud dataset, for various values of the nearest neighbor hyperparameter $k$. Error bars denote standard errors, computed over the $100$ random training/validation/test splits.}
    \label{fig:credit_card_f1_over_k}
\end{figure}

The main question we sought to investigate here was whether the theoretical finding, in Theorem~\ref{thm:generalized_bayes}, that stochastic classification is sometimes necessary in order to obtain optimal prediction performance under general performance metrics, would be visible in real data. To investigate this, we partitioned the dataset randomly into a training subset ($60\%$ of samples), a validation subset ($20\%$ of samples), and a test subset ($20\%$ of samples).
We fit a $k$-nearest neighbor regressor (with Euclidean distance as the underlying metric) to the training subset, used the validation subset to select optimal deterministic and generalization thresholds, and then used the test subset to evaluate performance. We evaluated performance in terms of $F_1$ score, since it is perhaps the CMM most widely used with imbalanced datasets. We then repeated this experiment with $100$ random train/validation/test splits and report aggregate results over these independent trials.

We generally found that, as predicted by our theoretical results, stochastic thresholding generally outperforms deterministic thresholding by a small but consistent margin. Figure~\ref{fig:credit_card_f1_over_class_imbalance} shows that, for fixed nearest neighbor hyperparameter $k = 4$, this effect is robust across differing degrees of class imbalance, for imbalance ratios ranging from $1:1$ (perfect balance) to $57:1$ (the full dataset), where class imbalance here was manipulated by down-sampling the negative class. Similarly, Figure~\ref{fig:credit_card_f1_over_k} shows that this effect is robust over different values of the nearest neighbor hyperparameter $k \in \{2, 4, 8, 16, 32, 64, 128\}$.